\documentclass{article}
\usepackage{mathdef}
\usepackage{array,multirow,booktabs}
\usepackage{arydshln}
\usepackage{hyperref}

\usepackage{enumitem}
\usepackage{authblk}
\newcommand{\MM}{\texttt{MM-M}} 
\newcommand{\MS}{\texttt{MS}} 
\newcommand{\ML}{\texttt{MM-L}} 

\newcommand{\Jacobi}{\texttt{JAC}}
\newcommand{\Annulus}{\texttt{ANN}}
\newcommand{\Curve}{\texttt{CUR}}
\newcommand{\Hole}{\texttt{HOL}}
\newcommand{\Map}{\texttt{MAP}}
\newcommand{\Torus}{\texttt{TOR}}

\newcommand{\bs}[1]{\boldsymbol{#1}}

\title{A Stieltjes algorithm for generating multivariate orthogonal polynomials}
\author[1]{Zexin Liu}
\author[2]{Akil Narayan}
\affil[1, 2]{Department of Mathematics, and Scientific Computing and Imaging (SCI) Institute, University of Utah}

\begin{document}
\maketitle

\begin{abstract}
  Orthogonal polynomials of several variables have a vector-valued three-term recurrence relation, much like the corresponding one-dimensional relation. This relation requires only knowledge of certain recurrence matrices, and allows simple and stable evaluation of multivariate orthogonal polynomials. In the univariate case, various algorithms can evaluate the recurrence coefficients given the ability to compute polynomial moments, but such a procedure is absent in multiple dimensions. We present a new Multivariate Stieltjes (\MS) algorithm that fills this gap in the multivariate case, allowing computation of recurrence matrices assuming moments are available. The algorithm is essentially explicit in two and three dimensions, but requires the numerical solution to a non-convex problem in more than three dimensions. Compared to direct Gram-Schmidt-type orthogonalization, we demonstrate on several examples in up to three dimensions that the \MS{} algorithm is far more stable, and allows accurate computation of orthogonal bases in the multivariate setting, in contrast to direct orthogonalization approaches.
\end{abstract}

\section{Introduction}
Orthogonal polynomials are a mainstay tool in numerical analysis and scientific computing, and serve as theoretical and computational foundations for numerical algorithms involving approximation and quadrature \cite{van2007discrete, gautschi_orthogonal_2004, gautschi2006orthogonal}.

It is well-known even in the multivariate setting that such families of polynomials satisfy three-term recurrence relations \cite{jackson1936formal,krall1967orthogonal,kowalski1982recursion,kowalski1982orthogonality,xu1993multivariate,xu1994multivariate,dunkl2014orthogonal}, which are commonly exploited for stable evaluation and manipulation of such polynomials. Identification or numerical approximation of the coefficients in such relations is therefore of great importance, and in the univariate setting many algorithms for accomplishing such approximations exist \cite{gautschi_orthogonal_2004,liu2021computation}. Such procedures are absent in the multivariate setting; this paper provides one algorithmic solution to fill this gap.

\subsection{Challenges with computing orthogonal polynomials}\label{ssec:intro:challenges}
Throughout, we assume the ability to compute generalized polynomial moments, i.e., there is some algorithm available to us that evaluates $p \mapsto \int_{\numset{R}^d} p(x) \dx{\mu}(x)$ for a given positive measure $\mu$ on $\numset{R}^d$. This assumption is required for univariate algorithms as well. 

With moment information, one could devise a linear algebraic scheme that orthogonalizes some known basis (say monomials) into an orthonormal basis, seemingly providing a solution to the evaluation of orthogonal polynomials. But in finite precision, even stable orthonormalization algorithms can be ineffective due to the high condition number of the map from integral moments to the orthonormal basis; we demonstrate this in Figure \ref{fig:cond}. Thus, even computational identification of an orthonormal polynomial basis is challenging, let alone computation of recurrence matrices.

Although computing moments with respect to fairly general multivariate measures $\mu$ is certainly an open challenge, it is not the focus of this article: We focus on the separate, open challenge of computing recurrence coefficients (allowing stable evaluation of multivariate polynomials) \textit{given} the ability to compute moments.

\subsection{Contributions}
The main contribution of this paper is to extend existing methods for computing recurrence coefficients from the univariate case to the multivariate case. Recognizing that the \textit{Stieltjes} algorithm for computing recurrence coefficients in univariate problems has shown tremendous success \cite{stieltjes1884quelques,stieltjes1884some,gautschi1984some}, we devise a new, \textit{Multivariate Stieltjes} (\MS) algorithm for computing recurrence coefficients (matrices in the multivariate setting), and hence also for computing a multivariate orthonormal basis. Thus, our contribution, the Multivariate Stieltjes algorithm, is a new method for tackling the challenge identified in Section \ref{ssec:intro:challenges}. We demonstrate with several numerical examples the (substantially) improved stability of the \MS{} algorithm compared to alternative Gram-Schmidt-type approaches for computing an orthonormal basis.

The tools we employ are, qualitatively, direct multivariate generalizations of existing univariate ideas. However, the technical details in the multivariate case are so disparate from the univariate case that we must employ somewhat different theories and develop new algorithms. Our \MS{} algorithm has explicit steps in two and three dimensions, but requires non-convex optimization in four or more dimensions. We first review some existing methods to compute univariate recurrence coefficients (see Section \ref{ssec:uni-op}) and introduce notation, properties, and the three-term relation for multivariate polynomials in Section \ref{ssec:multi}. In Section \ref{sec:non-uniqueness}, we propose a \textit{canonical basis} that identifies a computational strategy for direct evaluation of multivariate polynomials. We follow this by Section \ref{sec:tensor-weight}, which shows that if $\mu$ is tensorial, then a tensor-product basis is in fact a canonical basis. Algorithms are discussed in Section \ref{sec:non-tensorial}; Section \ref{ssec:moment} describes a direct procedure using orthonormalization given polynomial moments. The new multivariate Stieltjes procedure is described in Section \ref{ssec:stieltjes}. Finally, we present a wide range of numerical examples in Section \ref{sec:numerical}, which compares these approaches, and demonstrates the improved accuracy of the \MS{} procedure.

We mention that our goals are similar to the results in \cite{barrio_three_2010,waldron_recursive_2011}, which produce explicit recurrence relations. However, these results are either specialized to certain domains, or use recurrence matrices as known ingredients. Our procedures compute recurrence matrices for general measures, and hence are quite different.

\subsection{Assumptions and caveats}
Throughout this manuscript, we assume that integral moments of arbitrary polynomials are available/computable. For ``simple'' domains, we realize this through mapped/tensorized quadrature (which is sometimes exact and sometimes approximate). For more complicated domains, we discretize the measure $\mu$ as the empirical measure associated to a large number of realizations that are independently and identically distributed according to $\mu$. Thus, sometimes our numerical examples compute orthogonal polynomials with respect to an approximate measure. However, we use a sufficiently dense grid that such approximation error is relatively small. We emphasize that this approximation error is not the focus of this article; our goal is to devise a scheme that, given the ability to compute moments, accurately computes an orthonormal polynomial basis.

The new \MS{} algorithm we develop is effective compared to direct orthonormalization schemes when the condition number of the Gram moment matrix is large. For, e.g., small dimensions $d$ and polynomial degree, this moment matrix typically is not too ill-conditioned, and so there is little benefit in the \MS{} algorithm for such situations. However, when one requires polynomials of moderately large degree, or when the Gram matrix is extremely ill-conditioned, we show that the \MS{} algorithm is effective.

Finally, we note that the technique we present leverages theories associated with total-degree spaces of polynomials, and does not directly apply to more exotic spaces. In particular, we assume that $\mu$ is non-degenerate with respect to $d$-variate polynomials. This is ensured if, for example, $\mu$ has a positive Lebesgue density over any open ball in $\numset{R}^d$.


\section{Background and notation}
We use the standard multi-index notation in $d \in \numset{N}$ dimensions.  With $\numset{N}_0$ the set of nonnegative integers, a multi-index in $d$ dimensions is denoted by $\alpha = (\alpha_1, \ldots, \alpha_d) \in \numset{N}_0^d$. For $\alpha \in \numset{N}_0^d$, and $x = (x_1, \ldots, x_d) \in \numset{R}^d$, we write monomials as $x^\alpha = x_1^{\alpha_1}, \ldots, x_d^{\alpha_d}$. The number $|\alpha| \coloneqq \alpha_1 + \cdots+ \alpha_d$ is the degree of $x^\alpha$. We denote the space of $d$-variate polynomials of exactly degree $n \in \numset{N}_0$, and up to degree $n$, respectively, by
\begin{align*}
  \numset{P}_n^d &\coloneqq \spn \{x^\alpha: |\alpha| = n, \alpha \in \numset{N}_0^d\}, &
  \Pi_n^d &\coloneqq \spn \{x^\alpha: |\alpha| \le n, \alpha \in \numset{N}_0^d\},
\end{align*}
The dimensions of these spaces are, respectively,
\begin{align*}
  r_n = r_n^d &\coloneqq \dim \numset{P}_n^d = \binom{n+d-1}{n}, & R_n = R_n^d &\coloneqq \dim \Pi_n^d = \binom{n+d}{n} = \sum_{j=0}^n r_j,
\end{align*}
Since the dimension $d$ will be arbitrary but fixed in our discussion, we will frequently suppress notational dependence on $d$ and write $r_n$, $R_n$. We will also require differences between dimensions of subspaces,
\begin{align*}
  \Delta r_n &\coloneqq r_n - r_{n-1} = \binom{n+d-2}{n}, & n &\geq 0, & d &\geq 2,
\end{align*}
where we define $r_{-1} \coloneqq 0$. We will not be concerned with $\Delta r_n$ when $d = 1$, although one could define $\Delta r_n = 0$ in this case.

Throughout we assume that $\mu$ is a given positive measure on $\numset{R}^d$, with $d \in \numset{N}$. The support of $\mu$ may be unbounded. We assume that, given any non-trivial polynomial $p$, we have, 
\begin{align}\label{eq:mu-nondegenerate}
  0 < \inprod{p}{p} &< \infty, &
  \inprod{f}{g} \coloneqq \int f(x) g(x) \dx{\mu}(x).
\end{align}
and we implicitly assume that $\int q(x) \dx{\mu}(x)$ is computationally available for arbitrary polynomials $q$. Our main goal is to compute an orthonormal polynomial basis, where orthonormality is defined through the inner product above.

\subsection{Univariate orthogonal polynomials}\label{ssec:uni-op}
In the univariate case, let $\mu$ be a positive Borel measure on $\numset{R}$ with finite moments. 
A standard Gram-Schmidt process applied to monomials with the inner product $\langle \cdot,\cdot\rangle$ yields a sequence of orthonormal polynomials $\{p_n(x)\}_{n=0}^\infty$, which satisfies $\langle p_n, p_m \rangle = \delta_{m,n}$, where $\delta_{m,n}$ is the Kronecker delta. In practical settings, it is well-known that utilizing the three-term recurrence formula for evaluation is more computationally stable compared to direct orthonormalization techniques. There exist coefficients $b_0$ and $\{a_n, b_n\}_{n \in \numset{N}}$, with $a_n = a_n(\mu)$ and $b_n = b_n(\mu)$, such that 
\begin{align}\label{eq:univariate-ttr}
x p_n(x) &= b_{n+1} p_{n+1}(x) + a_{n+1} p_n(x) + b_n p_{n-1}(x), & n &\geq 0,
\end{align}
where $b_n > 0$ for all $n \in \numset{N}_0$, and $p_0 = 1/b_0$, $p_{-1} \equiv 0$. Availability of the coefficients $(a_n, b_n)$ not only enables evaluation via the recurrence above, but also serve as necessary ingredients for various computational approximation algorithms, e.g., quadrature. Thus, knowledge of these coefficients is of great importance in the univariate case. In some cases these coefficients are explicitly known \cite{szego_orthogonal_1975}, and in cases when they are not, several algorithms exist to accurately compute approximations \cite{chebyshev1859interpolation,rutishauser1963jacobi,sack_algorithm_1971,wheeler_modified_1974,gautschi_survey_1981,gragg1984numerically,gautschi_orthogonal_2004,liu2021computation}. The procedure we present in this paper is generalization of the (univariate) Stieltjes procedure \cite{stieltjes1884quelques,stieltjes1884some,gautschi1984some}. 

\subsection{Multivariate orthogonal polynomials}\label{ssec:multi}
Now let $\mu$ be a measure on $\numset{R}^d$. 
Again a Gram-Schmidt process applied to the multivariate monomials $x^\alpha$, $\alpha \in \numset{N}_0^d$ produces a sequence of orthogonal polynomials in several variables. 
However, in the multivariable setting it is more natural to proceed in a degree-graded fashion: We shall say that $p \in \Pi_n^d$ is an orthogonal polynomial of degree $n > 0$ with respect to $\dx{\mu}$ if
\begin{align*}
  \inprod{p}{q} &= 0, & \forall q &\in \Pi_{n-1}^d.
\end{align*}
Of course, $p \equiv 1$ is the unique unit-norm degree-$0$ polynomial with positive leading coefficient. The definition above allows us to introduce, $\funspace{V}_n^d$, the space of orthogonal polynomials of degree of exactly $n$; that is
\begin{align*}
\funspace{V}_n^d = \{p \in \Pi_n^d: \langle p, q \rangle = 0, \forall q \in \Pi_{n-1}^d\}.
\end{align*}
Our assumption \eqref{eq:mu-nondegenerate} on non-degeneracy of $\mu$ implies that $\dim \funspace{V}_n^d = \dim \numset{P}_n^d = r_n$. This allows us to state results in terms of size-$r_n$ vector functions containing orthonormal bases of $\funspace{V}_n^d$ for $n \geq 0$. We fix an(y) orthonormal basis for each $n \in \numset{N}_0$,
\begin{align}\label{eq:pn-def}
  \mathbbm{p}_n &= (p_{(n,j)})_{j=1}^n = (p_{(n, 1)}, \ldots, p_{(n, r_n)})^T, & \mathrm{span}\{p_{(n,j)}\}_{j=1}^{r_n} = \mathcal{V}_n^d,
\end{align}
This fixed basis $\mathbbm{p}_n$ has a three-term recurrence relation analogous to \eqref{eq:univariate-ttr}: There exist unique matrices $A_{n+1,i} \in \numset{R}^{r_n \times r_n}$, $B_{n+1,i} \in \numset{R}^{r_n \times r_{n+1}}$, such that
\begin{align}\label{eq:multi-ttr}
    x_i \mathbbm{p}_n(x) &= B_{n+1,i} \mathbbm{p}_{n+1}(x) + A_{n+1,i}\mathbbm{p}_n(x) + B_{n,i}^T \mathbbm{p}_{n-1}(x), & i \in [d].
\end{align}
where we define $\mathbbm{p}_{-1} = 0$ and $\mathbbm{p}_{0}(x) = 1$. These matrices must satisfy the conditions,
\begin{align}\label{eq:rank-condition}
  \mathrm{rank}(B_{n,i}) &= r_{n-1}, & \mathrm{rank}(B_n) &= r_n, & B_n &\coloneqq \left( B^T_{n,1},\; \ldots,\; B^T_{n,d}\right)^T \in \numset{R}^{d r_{n-1} \times r_n},
\end{align}
see \cite[Theorem 3.3.4]{dunkl2014orthogonal}.

Given an orthonormal basis $\{\mathbbm{p}_n\}_{n \geq 0}$, there are certain recurrence matrices that make the relation \eqref{eq:multi-ttr} true. However, polynomials generated by \eqref{eq:multi-ttr} for an arbitrary set of matrices $A_{n,i}$ and $B_{n,i}$ need not be orthogonal polynomials. The following are a set of necessary conditions, the \textit{commuting conditions}, that we will require.
\begin{theorem}[{\cite[Theorem 2.4]{wolfgang_zu_castell_inzell_2005}}]\label{thm:commuting}
  If $A_{n,i}$ and $B_{n,i}$ for $i \in [d]$ and $n \in \numset{N}_0$ are recurrence matrices corresponding to an orthonormal polynomial sequence, then they satisfy the following conditions for every $i, j \in [d]$ and $n \in \numset{N}$:
  {\small 
  \begin{subequations}\label{eq:cc}
    \begin{align}\label{eq:cc-1}
      B_{n+1,i} B_{n+1,j}^T + A_{n+1,i} A_{n+1,j} + B_{n,i}^T B_{n,j} &=B_{n+1,j} B_{n+1,i}^T + A_{n+1,j} A_{n+1,i} + B_{n,j}^T B_{n,i}, \\
    B_{n,i} A_{n+1,j} + A_{n,i} B_{n,j} &=
    B_{n,j} A_{n+1,i} + A_{n,j} B_{n,i}, & 
      \\ \label{eq:cc-3}
    B_{n,i} B_{n+1,j} &= B_{n,j} B_{n+1,i}, 
  \end{align}
  \end{subequations}
  }
  where \eqref{eq:cc-1} holds for $n=0$ also.
\end{theorem}
The conditions above can be derived by expressing certain matrix moments using two different applications of \eqref{eq:multi-ttr}. More informally, \eqref{eq:multi-ttr} for a fixed $n$ corresponds to $d$ sets of conditions of size $r_n$ that determine the $r_{n+1}$ degrees of freedom in $\mathbbm{p}_{n+1}$. In order for all these conditions to be consistent, the recurrence matrices must satisfy certain constraints, namely \eqref{eq:cc}. 

One starting point for our multivariate Stieltjes algorithm is the following direct observation: Inspection of the recurrence relation \eqref{eq:multi-ttr} reveals that the coefficient matrices can be computed via:
\begin{equation}\label{eq:AB-def}
    A_{n+1,i} = \int_{\numset{R}^d} x_i \mathbbm{p}_n(x) \mathbbm{p}_n^T(x) \dx{\mu}, \qquad B_{n+1,i} = \int_{\numset{R}^d} x_i \mathbbm{p}_n(x) \mathbbm{p}_{n+1}^T(x) \dx{\mu}.
\end{equation}
Note that $A_{n+1,i}$ is determined by (quadratic moments of) $\mathbbm{p}_n$, but $B_{n+1,i}$ does not have a similar characterization using only degree-$n$ orthogonal polynomials. 



\section{Evaluation of polynomials}\label{sec:non-uniqueness}
The three-term relation \eqref{eq:multi-ttr} does not immediately yield an evaluation scheme. We discuss in this section one approach for such an evaluation, which prescribes a fixed orthonormal basis. (That is, we remove ``most'' of the unitary equivalence freedom for $\mathbbm{p}_n$.) Our solution for this is introduction of a particular ``canonical'' form.


\subsection{Canonical bases}\label{ssec:canonical}
The three-term recurrence \eqref{eq:multi-ttr} for a fixed $i \in [d]$ is an underdetermined set of equations for $\mathbbm{p}_{n+1}$, and hence this cannot be used in isolation to evaluate polynomials. To make the system determined, one could consider \eqref{eq:multi-ttr} for \textit{all} $i \in [d]$ simultaneously. To aid in this type of procedure, we make a special choice of orthonormal basis that we will see amounts to choosing a particular sequence of unitary transformations.
\begin{definition}\label{def:canonical}
  Let $\{\mathbbm{p}_n\}_{n \in \numset{N}_0}$ be an orthonormal set of polynomials with recurrence matrices $A_{n,i}$ and $B_{n,i}$ for $i \in [d]$, $n \in \numset{N}_0$. 
  We say that $\{\mathbbm{p}_n\}_{n \in \numset{N}_0}$ is a \textit{canonical} (orthonormal) basis, and that the matrices $A_{n,i}$ and $B_{n,i}$ are in \textit{canonical} form if the following is true: For every $n \in \numset{N}$ we have the condition
      \begin{align}\label{eq:BTB-def}
        B_{n}^T B_{n} = \sum_{i \in [d]} B_{n,i}^T B_{n,i} = \Lambda_n,
      \end{align}
    where $B_n$ is as defined in \eqref{eq:rank-condition}, and the matrices $\{\Lambda_n\}_{n \in \numset{N}}$ are a sequence of diagonal matrices with the elements of each matrix appearing in non-decreasing order.
\end{definition}
Although the condition for canonical form appears only explicitly through a condition on $B_{n,i}$, the matrices $A_{n,i}$ are coupled with $B_{n,i}$ through the commuting conditions in \eqref{thm:commuting} so that canonical form is implicitly a condition on the $A_{n,i}$ matrices as well. The utility of a canonical basis is revealed by the following result.
\begin{theorem}[See also {\cite[Theorem 3.3.5]{dunkl2014orthogonal}}]\label{thm:diagonal-evaluation}
  Let the orthonormal basis $\{\mathbbm{p}_n\}_{n \in \numset{N}_0}$ be a canonical basis so that the associated matrices $A_{n,i}$ and $B_{n,i}$ satisfy \eqref{eq:BTB-def}. Then,
  {\small
  \begin{align}\label{eq:diagonal-ttr}
    \Lambda_{n+1} \mathbbm{p}_{n+1} = \left( \sum_{i \in [d]} x_i B_{n+1,i}^T \right) \mathbbm{p}_n - \left( \sum_{i \in [d]} B_{n+1,i}^T A_{n+1,i} \right) \mathbbm{p}_n - \left( \sum_{i \in [d]} B_{n+1,i}^T B_{n,i}^T \right) \mathbbm{p}_{n-1}, 
  \end{align}
  }
  for each $n \geq 0$, where $\Lambda_{n+1} \in \numset{R}^{r_{n+1} \times r_{n+1}}$ is diagonal and positive-definite (and hence invertible).
\end{theorem}
\begin{proof}
  The relation \eqref{eq:diagonal-ttr} is computed by (vertically) stacking the $d$ relations \eqref{eq:multi-ttr} for $i \in [d]$. Although this stacked system is overdetermined, it is consistent since \eqref{eq:multi-ttr} must hold for all $i \in [d]$. Multiplying both sides of the stacked system by $B_{n+1}^T$ yields \eqref{eq:diagonal-ttr}, and we need only explain why $\Lambda_{n+1} = B_{n+1}^T B_{n+1}$ is diagonal and positive-definite. The diagonal property is immediate since the basis is canonical, and clearly is positive semi-definite. That it is in fact positive-definite is a consequence of the rank condition \eqref{eq:rank-condition}, ensuring that $r_{n+1} = \rank(B_{n+1}) = \rank(\Lambda_{n+1})$.
\end{proof}
Equation \eqref{eq:diagonal-ttr} demonstrates how knowledge of $(A_{n+1,i}, B_{n+1,i})_{i \in [d]}$ translates into direct evaluation of $\mathbbm{p}_{n+1}$: the right-hand side of \eqref{eq:diagonal-ttr} is computable, and need only be scaled elementwise by the inverse diagonal of $\Lambda_{n+1}$. The main requirement for this simple technique for evaluation is that the recurrence matrices are in canonical form. Fortunately, it is fairly simple to transform any valid recurrence matrices into canonical form.

\subsection{Transformation to canonical form}\label{ssec:transformation}
There is substantial freedom in how the basis elements of the vectors $\mathbbm{p}_n$ are chosen.  
In particular, let $\left\{U_n\right\}_{n \geq 0}$ be an arbitrary family of orthogonal matrices defining a new basis $\mathbbm{q}_n$,
\begin{align}\label{eq:qn-def}
  \mathbbm{q}_n(x) &\coloneqq U_n \mathbbm{p}_n(x), & U_n &\in \numset{R}^{r_n \times r_n}, & U_n^T U_n &= I.
\end{align}
A manipulation of \eqref{eq:multi-ttr} shows that the basis elements $\mathbbm{q}_n$ satisfy a three-term relation,
\begin{align*}
x_i \mathbbm{q}_n(x) &= D_{n+1,i} \mathbbm{q}_{n+1}(x) + C_{n+1,i}\mathbbm{q}_n(x) + D_{n,i}^T \mathbbm{q}_{n-1}(x), & i \in [d],
\end{align*}
where the new matrices $C_{n,i}$ and $D_{n,i}$ can be explicitly derived from the unitary matrices $U_n$ and the recurrence matrices for $\mathbbm{p}_n$,
\begin{align}\label{eq:CD-def}
  C_{n,i} &= U_{n-1} A_{n,i} U_{n-1}^T, & 
  D_{n,i} &= U_{n-1} B_{n,i} U_{n}^T, &
  (i, n) &\in [d] \times \numset{N}_0.
\end{align}
Our goal now is to take arbitrary valid recurrence matrices $(A_{n,i}, B_{n,i})$ and identify the unitary transform matrices $\{U_n\}_{n \in \numset{N}}$ so that $(C_{n,i}, D_{n,i})$ are in canonical form. Since $B_n^T B_n$ is symmetric, then it has an eigenvalue decomposition,
\begin{align}\label{eq:Lambdan}
  B_n^T B_n = V_n \Lambda_n V_n^T,
\end{align}
with diagonal eigenvalue matrix $\Lambda_n$ and unitary matrix $V_n$, where we assume the diagonal elements of $\Lambda_n$ are in non-increasing order. Then by defining,
\begin{align*}
  U_n &= V_n^T, & n &\geq 1,
\end{align*}
which identifies $C_{n,i}$ and $D_{n,i}$ through \eqref{eq:CD-def}, we immediately have that 
\begin{align*}
  \sum_{i \in [d]} D_{n,i}^T D_{n,i} = \sum_{i\in[d]} (V_n^T B_{n,i}^T V_{n-1}) (V_{n-1}^T B_{n,i} V_n) = V_n^T B_n^T B_n V_n = \Lambda_n,
\end{align*}
and hence $\mathbbm{q}_n$ is a canonical basis, with $C_{n,i}$ and $D_{n,i}$ the associated recurrence matrices in canonical form. Thus, for each fixed $n$, a transformation to canonical form is accomplished via a single size-$r_n$ symmetric eigenvalue decomposition.

The discussion above also reveals how much non-uniqueness there is in the choice of canonical form through non-uniqueness in the symmetric eigenvalue decomposition: each row of $U_n$ is non-unique up to a sign, and arbitrary unitary transforms of sets of rows corresponding to invariant subspaces of $B_n^T B_n$ are allowed. If the non-increasing diagonal elements of $\Lambda_n$ have distinct values, then each of the functions in the vector $\mathbbm{q}_n$ is unique up to a multiplicative sign.

\subsection{Three-term relation conditions on matrices}
Our computational strategy computes $A_{n,i}$, $B_{n,i}$ through manipulations of polynomials computed via \eqref{eq:diagonal-ttr}, which are assumed to be orthonormal polynomials. One subtlety is that if we prescibe recurrence matrices (through a computational procedure), then usage of \eqref{eq:diagonal-ttr} requires additional conditions to be equivalent to \eqref{eq:multi-ttr}. This equivalence is summarized as follows, and forms the basis for the conditions we aim to impose in our algorithm.
\begin{theorem}[{\cite[Theorem 3.5.1]{dunkl2014orthogonal}}]\label{thm:matrix-conditions}
  Given matrices $A_{n,i}, B_{n,i}$, let $\mathbbm{p}_n$, $n\geq 0$ be generated through \eqref{eq:diagonal-ttr}. Then $\{\mathbbm{p}_n\}$ is orthonormal with respect to some positive-definite bilinear functional on $\numset{R}^d$ if and only if:
  \begin{itemize}[itemsep=0pt]
    \item $A_{n,i}$ is symmetric for all $(i,n) \in [d] \times \numset{N}_0$.
    \item $B_{n,i}$ satisfies the rank condition in \eqref{eq:rank-condition}: $\mathrm{rank}(B_{n,i}) = r_{n-1}$.
    \item The matrices satisfy the commuting conditions \eqref{eq:cc}.
  \end{itemize}
\end{theorem}
Additional conditions are required to ensure that the stated positive-definite functional is integration with respect to a measure $\mu$, but since our algorithm considers only finite $n$, such a distinction is not needed for us.

\section{Recurrence matrices and canonical form for tensorial measures}\label{sec:tensor-weight}
We take a slight detour in this section to identify recurrence matrices and associated canonical forms for tensorial measures $\mu$.  If $\mu$ is a tensorial measure, i.e., $\mu = \otimes_{j=1}^d \mu_j$,
where each $\mu_j$ is a positive measure on $\numset{R}$ satisfying the non-degeneracy condition \eqref{eq:mu-nondegenerate}, then there exists a sequence of univariate orthonormal polynomials for each $j \in [d]$. In particular, fixing $j$, there exist coefficients $\{a_{j,n}\}_{n \in \numset{N}} \subset \numset{R}$ and $\{b_{j,n}\}_{n \in \numset{N}_0} \subset (0, \infty)$ such that the sequence of polynomials defined by the recurrence,
\begin{align*}
  x_j p_{j,n}(x_j) &= b_{j,n+1} p_{j,n+1}(x_j) + a_{j,n} p_{j,n}(x_j) + b_{j,n} p_{j,n-1}(x_j), & n & \geq 0,
\end{align*}
are $L^2_{\mu_j}(\numset{R})$-orthonormal,
With these univariate polynomials, we can directly construct multivariate orthonormal polynomials via tensorization,
\begin{align}\label{eq:p-tensorized}
  p_{\alpha}(x) &\coloneqq \prod_{j=1}^d p_{j,\alpha_j}(x_j), & \alpha &\in \numset{N}_0^d,
\end{align}
where $\deg p_{\alpha} = |\alpha|$. We now construct vectors containing polynomials of a specific degree as in \eqref{eq:pn-def}. With $J_n$ denote any ordered set of the multi-indices $\alpha$ satisfying $|\alpha| = n \geq 0$, define,
\begin{align}\label{eq:pn-tensorized}
  \mathbbm{p}_n &= \left(p_{\alpha^{(n,1)}}, \ldots, p_{\alpha^{(n,r_n)}}\right)^T, & 
  J_n &\coloneqq \left( \alpha^{(n,1)}, \ldots, \alpha^{(n, r_n)} \right), 
\end{align}
where $|\alpha^{(n,k)}| = n$ for each $k \in [r_n]$. 
Clearly this set $\{\mathbbm{p}_n\}_{n \in \numset{N}_0}$ is a sequence of multivariate orthogonal polynomials. The next subsection explicitly computes the associated recurrence matrices.

\subsection{Recurrence matrices}\label{ssec:recurrence}
Given the ordering of multi-indices defined by the sets $(J_n)_{n \in \numset{N}_0}$, we require a function that identifies the location of the index $\alpha^{(n,k)} + e_j$ in the set $J_{n+1}$. Fixing the sets $\{J_n\}_{n \geq 0}$, we define a function $c : \numset{N}_0^d \times [d] \rightarrow \numset{N}$ as,
\begin{align*}
  c\left(\alpha^{(n,k)}, i\right) = \textrm{The index $q \in [r_{n+1}]$ such that } \alpha^{(n,k)} + e_i = \alpha^{(n+1,q)}.
\end{align*}
Note that $c(\alpha, i)$ is well-defined for all $\alpha \in \numset{N}_0^d$ and $i \in [d]$. We can now identify recurrence matrices for the polynomials $\mathbbm{p}_n$ explicitly in terms of univariate recurrence coefficients.
\begin{theorem}\label{thm:ttr}
  With the set of polynomials $\left(\mathbbm{p}_n\right)_{n=0}^\infty$ defined by \eqref{eq:pn-tensorized} and \eqref{eq:p-tensorized}, then the recurrence matrices are given by,
  \begin{align}\label{eq:AB-tensorial}
    A_{n+1,i} &= \mathrm{diag}\left(a_{i, \alpha^{(n,k)}_i+1}\right)_{k=1}^{r_n}, &
    B_{n+1,i} &= \sum_{k=1}^{r_n} b_{i,\alpha_i^{(n,k)}+1} e_{r_n, k}e_{r_{n+1},c\left(\alpha^{(n,k)}, i\right)}^T,
  \end{align}
  where $\alpha^{(n,k)}_i$ is the $i$th component of $\alpha^{(n,k)}$
  and $e_{m,j}$ is an $m$-dimensional vector with entry $1$ in location $j$ and zeros elsewhere.
\end{theorem}
\begin{proof}
  Fixing any $n \in \numset{N}_0$ and $i \in [d]$, and for each $k \in [r_n]$, the $k$th component of the vector $\mathbbm{p}_n$ satisfies,
  \begin{align*}
    x_i p_{\alpha^{(n,k)}}(x) &= \left(x_i p_{\alpha^{(n,k)}_i}(x_i) \right) \prod_{j \neq i} p_{\alpha^{(n,k)}_j}(x_j)  \\
    &= b_{i,n+1} p_{\alpha^{(n,k)} + e_{d,i}}(x) + a_{i,n+1} p_{\alpha^{(n,k)}}(x) + b_{i,n} p_{\alpha^{(n,k)} - e_{d,i}}(x).
  \end{align*}
  By comparing the final expression above with the three-term recurrence \eqref{eq:multi-ttr}, the components of the matrices $A_{n+1,i}$ and $B_{n+1,i}$ can be identified and are given as in \eqref{eq:AB-tensorial}.
\end{proof}
When $\mu$ is tensorial, then computational methods for evaluating $\mathbbm{p}_n$ need not explicitly use recurrence matrices: the relations \eqref{eq:pn-tensorized} and \eqref{eq:p-tensorized} show that one needs only knowledge of \textit{univariate} recurrence coefficients $a_{j,n}, b_{j,n}$. Computation of these coefficients from the marginal measure $\mu_j$ is well-studied and has several solutions (cf. \eqref{ssec:uni-op}). However, if one wanted the recurrence matrices for this situation, they are explicitly available through the result above. We next show that, up to a permutation, such tensorized basis elements form a canonical basis.

\subsection{Canonical form for tensorial measures}
The explicit formulas in \eqref{eq:AB-tensorial} allow us to investigate the canonical condition \eqref{eq:BTB-def}:
\begin{align*}
  B_{n}^T B_{n} = \sum_{i \in [d]} B_{n,i}^T B_{n,i} = \diag\left( \sum_{i=1}^d b^2_{i, \alpha_i^{(n-1,1)}+1}, \ldots, \sum_{i=1}^d b^2_{i, \alpha_i^{(n-1,r_{n-1})}+1} \right)
\end{align*}
In order for $\mathbbm{p}_n$ to be a canonical basis, the diagonal elements above must be non-increasing. While it is not transparent on how to accomplish this for generic univariate coefficients $(a_{i,n}, b_{i,n})_{i \times n \in [d] \times \numset{N}}$, one can always computationally achieve this by reordering the elements of $\mathbbm{p}_n$, equivalently by reordering the elements of $J_n$, according to the elements on the right-hand side above.
\begin{corollary}
  Define permutation operators $P_n : [r_{n-1}] \rightarrow [r_{n-1}]$ given by,
  \begin{align*}
    P(k+1) \geq P(k) \hskip 10pt \Longrightarrow \hskip 10pt \sum_{i=1}^d b^2_{i, \alpha_i^{(n-1,k+1)}+1} \geq \sum_{i=1}^d b^2_{i, \alpha_i^{(n-1,k)}+1},
  \end{align*}
  for each $k \in [r_n]$. Defining a new basis, $\mathbbm{q}_n \coloneqq \left( p_{\alpha^{(n,P(k))}} \right)_{k=1}^{r_n}$, 
  then $\{\mathbbm{q}_n\}_{n=0}^\infty$ is a canonical basis.
\end{corollary}
We re-emphasize that identifying a canonical basis through an algorithmic version of the procedure above is not computationally advantageous compared to direct usage of \eqref{eq:pn-tensorized} and \eqref{eq:p-tensorized}. However, the procedure above gives us a method for oracle evaluation of recurrence matrices associated to a tensorial measure $\mu$.

\section{Algorithms for computing recurrence coefficient matrices}\label{sec:non-tensorial}
We discuss two strategies for computing the recurrence matrices $(A_{n,i}, B_{n,i})$. The first is a straightforward Moment Method (\MM) that directly uses moments (e.g., of monomials) to compute an orthonormal basis that can be used directly with \eqref{eq:AB-def} to compute the matrices. However, this approach is very ill-conditioned for moderately large degree and hence has limited utility. The second approach we present is the main novel advancement of this paper, a multivariate Stieltjes (\MS) algorithm that computes the recurrence matrices through an alternative procedure. 

\subsection{The moment method (\MM)}\label{ssec:moment}
A straightforward way to compute recurrence coefficients in the univariate case is to first perform an orthogonalization step to numerically generate orthogonal polynomials as linear expansions in some specified basis (say monomials), and second to exploit the linear expansion expressions to compute the recurrence coefficients. This works in the multivariate setting as well. 
We suppose that a(ny) polynomial basis $\left\{\phi_j\right\}_{j \in \numset{N}}$ is given with the properties,
\begin{align*}
  \spn \left\{ \phi_j \right\}_{j=1}^{R_n} = \Pi_n^d, \hskip 5pt & \hskip 15pt n \in \numset{N}_0. \\
  r_n < j \leq r_{n+1} \hskip 5pt &\Longrightarrow \hskip 5pt \deg \phi_j = n+1.
\end{align*}
Again, a simple example is that $\phi_j(x)$ is a multivariate monomial $x^\alpha$ for some total ordering of the multi-indices $\alpha$ that respects the partial order induced by the $\ell^1(\numset{N}_0^d)$ norm\footnote{Degree-graded lexicographic ordering over multi-indices is an example of one such ordering.}. We assume that quadratic $\phi_j$ moments are available, allowing us to compute a Gram matrix, 
\begin{align}\label{eq:gram-matrix}
  (G_n)_{i,j} &= m(i,j), & G_n &\in \numset{R}^{R_n \times R_n}, &
  m(i,j) &= \int \phi_i(x) \phi_j(x) \dx{\mu}(x). 
\end{align}
Given this matrix, we can compute monomial expansion coefficients for the orthonormal basis $\mathbbm{p}_n$:
\begin{align}\label{eq:pn-moment}
  G_n &= L_n L_n^T, & 
  \left(\mathbbm{p}_0, \; \ldots, \; \mathbbm{p}_n\right)^T &= L^{-1}_n \Phi_n, & 
  \Phi_n &= \left(\phi_1, \; \ldots, \; \phi_{R_n}\right)^T \in \numset{R}^{R_n},
\end{align}
where $\mathbbm{p}_j$ is an $r_j$-vector containing degree-$j$ orthonormal polynomials. Combining this with \eqref{eq:AB-def}, the recurrence matrices can be directly computed,
\begin{align}\label{eq:Bnp1-mm}
  A_{n+1,i} &= \widetilde{L}_n^{-1} G_{n,i} \widetilde{L}_n^{-T},  & 
  B_{n+1,i} &= \widetilde{L}_n^{-1} \widetilde{G}_{n+1,i} \widetilde{L}_{n+1}^{-T},
  & i &\in [d],
\end{align}
where $\widetilde{L}_{n}^{-1}$ is the $r_n \times R_n$ matrix formed from the last $r_n$ rows of $L_n^{-1}$, $G_{n,i} \in \numset{R}^{R_n \times R_n}$ has entries
\begin{align}\label{eq:Gnq}
  (G_{n,i})_{i,j} &= \int x_i \phi_i(x) \phi_j(x) \dx{\mu}(x), & i, j &\in [R_n],
\end{align}
and $\widetilde{G}_{n+1,i} \in \numset{R}^{R_n \times R_{n+1}}$ equals the first $R_n$ rows of $G_{n+1,i}$. Thus, so long as the polynomial moments $m(i,j)$ and \eqref{eq:Gnq} can be evaluated, then this allows direct computation of the recurrence matrices. Of course, if only an orthonormal basis (without recurrence matrices) is desired, then one may stop at \eqref{eq:pn-moment}.

The main drawbacks to the procedures above stem from accuracy concerns due to ill-conditioning of the $G_n$ matrices. If the $\phi_i$ are selected as ``close'' to an orthonormal basis, then the $G_n$ matrices can be well-conditioned, but \textit{a priori} knowledge of such a basis is not available in general scenarios. However, this method is flexible in the sense that with minor modifications one may compute an orthonormal basis (but not necessarily recurrence matrices) for very general, non-total-degree, polynomial spaces, such as hyperbolic cross spaces.

\subsection{The Multivariate Stieltjes algorithm (\MS)}\label{ssec:stieltjes}
In this section we describe a Stieltjes-like procedure for computing recurrence matrices, which partially overcomes ill-conditioning issues of the moment method. Like the univariate procedure, we directly compute the recurrence matrices instead of attempting to orthogonalize the basis, and the procedure is iterative on the degree $n$. Thus, throughout this section we assume that the recurrence matrices $\left\{A_{m,i}, B_{m,i} \right\}_{i \in [d], m \leq n}$ are available, and our goal is to compute $A_{n+1,i}$ and $B_{n+1,i}$ for $i \in [d]$. Our main strategy for accomplishing this will be to satisfy certain ``degree-$(n+1)$'' matrix moment conditions subject to the constraints identified in \ref{thm:matrix-conditions}.

The availability of the matrices for index $m \leq n$ implies that $\mathbbm{p}_m$ for all $m \leq n$ can be evaluated through the procedure in \ref{thm:diagonal-evaluation}. (In practice we transform $A_{m,i}, B_{m,i}$ for $m \leq n$ to be in canonical form.) We will compute $A_{n+1,i}$ directly, but compute factorized components and subblocks of $B_{n+1,i}$, and we identify those subblocks now. Consider the truncated singular value decomposition of $B_{n+1,i}$,
\begin{align}\label{eq:B-decomposition}
  B_{n+1,i} &= U_{n+1,i} \Sigma_{n+1,i} V^T_{n+1,i} = U_{n+1,i} \Sigma_{n+1,i} \left( \widehat{V}_{n+1,i}^T \;\; \widetilde{V}_{n+1,i}^T \right), & i &\in [d],
\end{align}
where $U_{n+1,i}$, $V_{n+1,i}$ have orthonormal columns and are of sizes,
\begin{align*}
  U_{n+1,i} &\in \numset{R}^{r_n \times r_n}, & \Sigma_{n+1,i} &\in \numset{R}^{r_n \times r_n}, & V_{n+1,i} &\in \numset{R}^{r_{n+1} \times r_n},
\end{align*}
and $\Sigma_{n+1,i}$ is diagonal with non-negative entries in non-increasing order. $\Sigma_{n+1,i}$ must be invertible due to \eqref{eq:rank-condition}. In \eqref{eq:B-decomposition} we have further decomposed the right-singular $V_{n+1,i}$ matrices into blocks, with $\widehat{V}$ the first $r_n$ rows, and $\widetilde{V}$ the remaining rows:
\begin{align}\label{eq:V-def}
  V_{n+1,i} &= \left(\begin{array}{c} \widehat{V}_{n+1,i} \\ \widetilde{V}_{n+1,i} \end{array}\right), & 
    \widehat{V}_{n+1,i} &\in \numset{R}^{r_n \times r_n} & \widetilde{V}_{n+1,i} &\in \numset{R}^{\Delta r_{n+1} \times r_n} 
\end{align}
We do not assume that the $B_{n+1,i}$ are in canonical form, and therefore have freedom to specify the unitary transform for degree $n+1$. Without loss we implicitly choose the unitary transform $U_{n+1}$ in \eqref{eq:CD-def} so that $V_{n+1,1}$ is equal to the first $r_n$ columns of the size-$r_{n+1}$ identity matrix. This uniquely identifies $V_{n+1,1}$,
\begin{align}\label{eq:V_1}
  \widehat{V}_{n+1,1} &= I_{r_n}, & \widetilde{V}_{n+1,1} &= \bm{0}.
\end{align}
We therefore need only compute $\widehat{V}_{n+1,i}, \widetilde{V}_{n+1,i}$ for $i \geq 2$.

The remainder of this section is structured as follows: We define certain moment matrices through a modified polynomials basis in Section \ref{sssec:moment-matrix}, which immediately yields the $A_{n+1,i}$ matrices. Section \ref{sssec:sym-moment} shows how we compute the $U$ and $\Sigma$ SVD matrices of $B$. We introduce ``mixed'' moments in Section \ref{sssec:mix-moment} that allow us to obtain the $\widehat{V}$ block of the $V$ matrices. The remaining block $\widetilde{V}$ is computed using different strategies depending on the dimension $d$ of the problem. For $d = 2$, Section \ref{sssec:ON-cond} shows that $\widetilde{V}$ can be computed almost directly. For $d \geq 3$ dimensions, we must enforce the commuting conditions \eqref{eq:cc}, which is somewhat easily done in $d=3$ dimensions, but requires nontrivial optimization for $d > 3$. 

\subsubsection{Moment matrices -- computing $A_{n+1,i}$}\label{sssec:moment-matrix}
The \MS{} algorithm begins by considering moments of polynomials that are not explicitly known \textit{a priori}, but are easily generated during an algorithmic procedure. In particular we introduce the moment matrices,
\begin{align}\label{eq:ST-def}
  S_{n,i} &\coloneqq \int x_i \mathbbm{p}_n \mathbbm{p}_n^T \dx{\mu}(x), & 
  T_{n,i,j} &\coloneqq \int \widetilde{\mathbbm{p}}_{n+1,i} \widetilde{\mathbbm{p}}_{n+1,j}^T \dx{\mu}(x),
\end{align}
both $r_n \times r_n$ matrices, 
where the modified polynomial basis $\widetilde{\mathbbm{p}}$ is defined as,
\begin{align}\label{eq:ptilde-def}
  \widetilde{\mathbbm{p}}_{n+1,i}(x) &\coloneqq x_i \mathbbm{p}_n - A_{n+1,i} \mathbbm{p}_n - B_{n,i}^T\mathbbm{p}_{n-1}.
\end{align}
Note that availability of $\left\{A_{m,i}, B_{m,i} \right\}_{i \in [d], m \leq n}$ along with the ability to evaluate $\mathbbm{p}_n$ imply that the moment matrices in \eqref{eq:ST-def} can be approximated via quadrature, just as is frequently done for the $\phi_j$ moments of Section \ref{ssec:moment}.

Inspection of \eqref{eq:AB-def} immediately reveals that,
\begin{align}\label{eq:An-stieltjes}
  A_{n+1,i} = S_{n,i},
\end{align}
and hence $A_{n+1,i}$ is directly computable. In addition, $A_{n+1,i}$ is symmetric (since $S_{n,i}$ is symmetric) and hence this satisfies the first condition in \ref{thm:matrix-conditions}. This then allows $\widetilde{\mathbbm{p}}_{n+1,i}$ to be evaluated, and hence allows $T_{n,i,j}$ to be computed. While evaluating $A_{n+1,i}$ is fairly straightforward from $S_{n,i}$, computing $B_{n+1,i}$ from $T_{n,i,j}$ is more involved. 
\subsubsection{Stieltjes symmetric moments -- computing $U_{n+1,i}, \Sigma_{n+1,i}$}\label{sssec:sym-moment}
The matrices $T_{n,i,i}$ of symmetric moments allow us to compute the $U$ and $\Sigma$ matrices in the SVD of $B$. A direct computation with the three-term recurrence \eqref{eq:multi-ttr} and the definition \eqref{eq:ST-def} reveals that
\begin{align}\label{eq:stieltjes-symmetric}
  T_{n,i,i} = B_{n+1,i} B^T_{n+1,i} = U_{n+1,i} \Sigma_{n+1,i}^2 U_{n+1,i}^T.
\end{align}
Therefore for each $i \in [d]$, we can first compute the square, symmetric matrix $T_{n,i,i}$, and subsequently its eigenvalue decomposition, ordering the eigenvalues in decreasing order. Then the eigenvector matrix of $T_{n,i,i}$ is $U_{n+1,i}$ and the square root of the eigenvalue matrix equals $\Sigma_{n+1,i}$. We can also conclude that $T_{n,i,i}$ is full-rank: Since $\mathbbm{p}_n$ contains linearly independent polynomials, then $x_i \mathbbm{p}_n$ contains linearly independent polynomials, and hence $\widetilde{\mathbbm{p}}_{n+1,i}$ contains linearly independent polynomials. Since $T_{n,i,i}$ is the Gram matrix for $\widetilde{\mathbbm{p}}_{n+1,i}$ and we have assumed $\mu$ is non-degenerate, then it must be of full rank. This observation then implies that $\Sigma_{n+1,i}$ is invertible, and hence that our computed $B_{n+1,i}$ is also full rank. Therefore, we have satisfied the second condition in \ref{thm:matrix-conditions}.
Now we are left only to compute the rectangular $V_{n+1,i}$ matrices, which will involve enforcing the third condition (the commuting conditions).

\subsubsection{Stieltjes mixed moments -- computing $\widehat{V}_{n+1,i}$}\label{sssec:mix-moment}
Using mixed moments $T_{n,i,j}$ with $i \neq j$, we can compute the square matrices $\widehat{V}_{n+1,i}$, which are subblocks of $V_{n+1,i}$. (Recall from \eqref{eq:V_1} that $V_{n+1,j}$ is already known for $j=1$, so we consider only $j > 2$.) 
A similar computation as the Stieltjes procedure in \eqref{eq:stieltjes-symmetric} yields that $T_{n,i,j} = B_{n+1,i} B^T_{n+1,j}$. By using the decomposition of $B_{n+1,i}$ in \eqref{eq:B-decomposition}, we conclude that 
\begin{align}\label{eq:stieltjes-mixed}
  \widehat{V}_{n+1,i}^T \widehat{V}_{n+1,j} + \widetilde{V}_{n+1,i}^T \widetilde{V}_{n+1,j} &= \Sigma_{n+1,i}^{-1} U_{n+1,i}^T T_{n,i,j} U_{n+1,j} \Sigma_{n+1,j}^{-1}, & i,j &\in [d].
\end{align}
Letting $i=1$ and utilizing \eqref{eq:V_1}, we have,
\begin{align}\label{eq:hatV}
  \widehat{V}_{n+1,j} &= \Sigma_{n+1,1}^{-1} U_{n+1,1}^T T_{n,1,j} U_{n+1,j} \Sigma_{n+1,j}^{-1}, & j &\geq 2,
\end{align}
where everything on the right-hand side is known and computable.

Note that here we have only utilized $T_{n,i,j}$ for $1 = i \neq j$. The case $1 \neq i \neq j$ is vacuous for $d=2$, and we will see that the remaining block $\widetilde{V}$ can already be computed. When $d \geq 3$, we do require $1 \neq i \neq j$ to identify $\widetilde{V}$.

\subsubsection{$d=2$: Orthonormality conditions for $\widetilde{V}_{n+1,i}$}\label{sssec:ON-cond}
For $d = 2$, we now need only compute $\widetilde{V}_{n+1,2}$. Since $\Delta r_{n} = 1$ for every $n$ when $d = 2$, then $\widetilde{V}_{n+1,2}$ is a $1 \times r_n$ vector. To reduce notational clutter, we consider fixed $n$ and use the following notation,
\begin{align*}
  \bm{y} \coloneqq \widetilde{V}^T_{n+1,2} \in \numset{R}^{r_n}.
\end{align*}
Then since $V_{n+1,2}$ has orthonormal columns, i.e., $V_{n+1,2}^T V_{n+1,2} = I_{r_n}$, we have,
\begin{align}\label{eq:V-orthonormality}
  \bm{y} \bm{y}^T = I_{r_n} - \widehat{V}_{n+1,2}^T \widehat{V}_{n+1,2},
\end{align}
which defines $\bm{y}$ up to a sign. More precisely, we have
\begin{align}\label{eq:yz-condition}
  \bm{y} = \pm \bm{z},
\end{align}
where $\bm{z}$ is computed either as a rank-1 Cholesky factor or as the positive semi-definite square root of the right-hand side of \eqref{eq:V-orthonormality}. Although it appears the multiplicative sign needs to be chosen, if we choose $(i,j) = (1,2)$ in \eqref{eq:cc-3}, then we have
\begin{align*}
  \left( B_{n,1} U_{n+1,2} \Sigma_{n+1,2} \widehat{V}_{n+1,2} \;\;\; \pm B_{n,1} U_{n+1,2} \Sigma_{n+1,2} \bm{z} \right) = B_{n,2} B_{n+1,1}.
\end{align*}
Thus the choice of sign makes no difference because the last column of $B_{n+1,1}$ is a zero vector. Therefore, we arbitrarily choose the sign. This completes the computation of $B_{n+1,i}$ for $d=2$; in this case we need not impose the commuting conditions, but they are needed for $d > 2$.

\subsubsection{$d > 2, n=0$: Falling back on moments}\label{sssec:backonmoments}
For $d > 2$ and $n = 0$, then $B_{n+1,i} \in \numset{R}^{1 \times d}$, and hence $\widetilde{V}_{n+1,i} \in \numset{R}^{1 \times (d-1)}$. We have two conditions to impose on this vector of length $d-1 \geq 2$:
\begin{itemize}
  \item The scalar-valued commuting condition \eqref{eq:cc-1} (the others do not apply for $n = 0$).
  \item A unit-norm condition on the column vector $V_{n+1,i}$ (as is used in the previous section)
\end{itemize}
This amounts to 2 conditions on this vector (although the second condition does not determine a multiplicative sign). However, as this approach combining all these conditions can be relatively cumbersome to simply determine a vector, in this case we fall back to using \texttt{MM} routines. When $n = 0$, the \texttt{MM} Gramians are typically well-conditioned. Therefore, when $d > 2$ and $n = 0$, we use \eqref{eq:Bnp1-mm} to compute the $B_{1,i}$ matrices.

\subsubsection{$d > 2$, $n > 0$: The commuting conditions -- computing $\widetilde{V}_{n+1,i}$}\label{sssec:c-cond}
We recall that our remaining task is to compute $\widetilde{V}_{n+1,i} \in \numset{R}^{\Delta_{n+1} \times r_n}$ for $2 \leq i \leq d$. Our tools to accomplish this will be (i) the commuting conditions (ii) orthonormality conditions on the columns of $V_{n+1,i}$, and (iii) the mixed moments $T_{n,i,j}$ for $1 \neq i \neq j$. For $n>0$, the commuting condition \eqref{eq:cc-3} with $i=1$ implies:
\begin{align}\label{eq:cc-d2-temp}
  B_{n,1} U_{n+1,j} \Sigma_{n+1,j} \left( \widehat{V}_{n+1,j}^T \;\; \widetilde{V}^T_{n+1,j} \right) &= B_{n,j} B_{n+1,1}, & j \geq 2,
\end{align}
where only $\widetilde{V}_{n+1,j}$ is unknown. Note that we have made the choice \eqref{eq:V_1} for $\widetilde{V}_{n+1,1}$, which implies that the last $\Delta r_{n+1}$ columns of $B_{n,j} B_{n+1,1}$ vanish, i.e., the last $\Delta r_{n+1}$ columns of \eqref{eq:cc-d2-temp} read,
\begin{align}\label{eq:Kj-def}
  K_{n+1,j} \widetilde{V}_{n+1,j} &= \bm{0}, & K_{n+1,j} &\coloneqq B_{n,1} U_{n+1,j} \Sigma_{n+1,j} \in \numset{R}^{r_{n-1} \times r_n}.
\end{align}
Thus, the columns of $\widetilde{V}_{n+1,j}$ lie in the kernel of the known matrix $K_{n+1,j}$. I.e., we have,
\begin{align}\label{eq:Cj-def}
  \widetilde{V}_{n+1,j}^T &= \Psi_j C_j, & \Psi_j &\in \numset{R}^{r_n \times \Delta r_n}, & C_j &\in \numset{R}^{\Delta r_n \times \Delta r_{n+1}},
\end{align}
where $C_j$ is unknown and $\Psi_j$ is known (computable), containing an orthonormal basis for $\ker(K_{n+1,j})$,
\begin{align*}
  \mathrm{range}(\Psi_j) &= \ker(K_{n+1,j}), & \Psi_j^T \Psi_j &= I_{\Delta r_n \times \Delta r_n}.
\end{align*}
We now use orthonormality of the columns of $V_{n+1,j}$. In particular this implies, 
\begin{align*}
  C_j C_j^T = D_j \coloneqq I_{\Delta r_n} - \Psi_j^T \widehat{V}_{n+1,j}^T \widehat{V}_{n+1,j} \Psi_j.
\end{align*}
Since $D_j$ is a symmetric, positive semi-definite matrix, then $C_j$ must be given by,
\begin{align}\label{eq:Cj-decomp}
  C_j &= E_j W_j, & E_j &\coloneqq \left( \sqrt{D_j} \;\;\;\;\;\; \bm{0}_{\Delta r_n \times (\Delta r_{n+1} - \Delta r_n)} \right), & W_j &\in \numset{R}^{\Delta r_{n+1} \times \Delta r_{n+1}},
\end{align}
where $W_j$ is a unitary matrix, and $\sqrt{D}_j$ is the symmetric positive semi-definite matrix square root of $D_j$. We therefore need only determine $W_j$. The final linear conditions we impose are the remaining mixed moment conditions from Section \ref{sssec:mix-moment} involving $T_{n,i,j}$ for $1 \neq i \neq j \geq 2$. Using \eqref{eq:stieltjes-mixed} and writing in terms of the unknown $W_i, W_j$ yields,
\begin{align}\label{eq:wopp}
  E_i W_i W_j^T E_j^T &= H_{n,i,j} & 2 \leq i, j \leq d,\;\; i \neq j \\\nonumber
  \textrm{subject to }\;  W_j^T W_j &= I_{\Delta r_{n+1}} & 2 \leq j \leq d.
\end{align}
where,
\begin{align*}
  H_{n,i,j} = \Psi_i^T (\Sigma_{n+1,i}^{-1} U_{n+1,i}^T T_{n,i,j} U_{n+1,j} \Sigma_{n+1,j}^{-1} - \widehat{V}_{n+1,i}^T \widehat{V}_{n+1,j}) \Psi_j, 
\end{align*}
is a computable matrix. This optimization problem for $\{W_j\}_{j=2}^d$ must be solved to determine the $W_j$ matrices. Once these are determined, then $\widetilde{V}_{n+1,i}$ is determined through \eqref{eq:Cj-def} and \eqref{eq:Cj-decomp}. This is a non-convex optimization, and is a generalization of a weighted orthogonal Procrustes problem (WOPP). Even the original WOPP has no known direct solution, so that numerical optimization must be employed to solve the above problem \cite{gower_procrustes_2004}. Fortunately, when $d=3$ some extra manipulations do yield a direct solution.

\subsubsection{$d=3$: Circumventing the WOPP}\label{sssec:d3-woop}
When $d = 3$, the problem \eqref{eq:wopp} simplifies substantially since we need only compute $W_2, W_3$. First we note that if any pair of orthogonal matrices $(W_2, W_3)$ satisfies \eqref{eq:wopp}, then so does the pair $(I_{\Delta r_{n+1}}, W_3 W_2^T)$. Therefore, we may choose $W_2 = I_{\Delta r_{n+1}}$ without loss. This then determines $\widetilde{V}_{n+1,2}$ through \eqref{eq:Cj-def} and \eqref{eq:Cj-decomp}.

The determination of $W_3$ in \eqref{eq:wopp} now reduces to an instance of a WOPP:
\begin{align}\label{eq:E3-wopp}
  E_2 W_3^T E_3^T = H_{n,2,3}, \hskip 10pt \textrm{subject to} \hskip 10pt W_3^T W_3 = I_{\Delta r_{n+1}} 
\end{align}
We now notice that $E_j$ defined in \eqref{eq:Cj-decomp} is rectangular, but when $d=3$ has only one more column than row, i.e., $\Delta r_n = n=1$ and thus $\Delta r_{n+1} = 1 + \Delta r_n$. Then with 
\begin{align*}
  E_j &= X_j \left( Y_j\;\; \bm{0}_{(n+1) \times 1} \right) Z_j^T,
\end{align*}
the reduced singular value decompositions of $E_2$ and $E_3$, then \eqref{eq:E3-wopp} can be rewritten as
\begin{align*}
  I_{(n+1) \times (n+2)} W I_{(n+2) \times (n+1)} &= Y_2^{-1} X_2^T H_{n,2,3} X_3 Y_3^{-1}, & 
  W &= Z_2^T W_3 Z_3,
\end{align*}
where $W$ is an orthogonal matrix of size $\Delta r_{n+1} = n+2$. Determining $W$ uniquely identifies $W_3$, but the above relation shows that the $(n+1) \times (n+1)$ principal submatrix of $W$ is given by $Y_2^{-1} X_2^T H_{n,2,3} X_3 Y_3^{-1}$. To determine the last row and column of $W$, we write,
\begin{align*}
  W &= \left(\begin{array}{cc} Y_2^{-1} X_2^T H_{n,2,3} X_3 Y_3^{-1} & \vline \\
    & \bm{v} \\
  \bm{w}^T & \vline \end{array}\right),  & \textrm{$\bm{w}$, $\bm{v}$ unknown},
\end{align*}
and since $W$ is an orthogonal matrix and must have orthonormal columns, then the first $n+1$ entries of $\bm{w}$ are determined up to a sign, and the signs can be determined by enforcing pairwise orthogonality conditions among the first $n+1$ columns. The final column $\bm{v}$ can be determined (up to an inconsequential sign) as a unit vector orthogonal to the first $n+1$ columns. Thus, $W$ is computable, which in turn determines $W_3$, and completes the solution to \eqref{eq:wopp} for $d=3$.

\section{Numerical test}\label{sec:numerical}
We present numerical results showcasing the effectiveness of the \MS{} algorithm for computing orthogonal polynomials on two and three-dimensional domains. 
We measure the efficacy of any particular method through a numerical orthogonality condition. Each method first generates computational approximations to the matrices $\{A_{n,i}, B_{n,i}\}$, and subsequently can generate computational approximations $\widehat{\mathbbm{p}}_n$ to an orthonormal basis $\mathbbm{p}_n$. Fixing a maximum degree $N$, we then use quadrature (approximate or exact as described) to evaluate the error matrix,
\begin{align}\label{eq:e_N}
  E &= \begin{pmatrix}
M_{0,0} & \cdots & M_{0,N} \\
\vdots & \ddots & \vdots \\
M_{N,0} & \cdots & M_{N,N} \\
\end{pmatrix} - I_{R_N}, & M_{m,n} &= \int \mathbbm{p}_{m}(x) \mathbbm{p}_{n}^T(x) \dx{\mu}(x),
\end{align}
where $E$ is an $R_N \times R_N$ matrix. We will both show plots of this error matrix and also report entrywise maximum values, $\|E\|_{\infty,\infty} = \max_{i,j \in [R_N]} |E_{i,j}|$. In all experiments, we set $N=39$ for $d=2$, and $N=15$ for $d=3$. (These choices result in $R_N = 820$ and $R_N = 816$, respectively.) We compare four methods:
\begin{itemize}
  \item (Exact) When $\mu$ is a tensor-product measure, we use the procedure in Section \ref{sec:tensor-weight} and summarized in Algorithm \ref{alg:ttr} to first explicitly compute the recurrence matrices, and subsequently to evaluate polynomials through the procedure in Section \ref{sec:non-uniqueness}.
  \item (\MS) The novel algorithm of this manuscript, the Multivariate Stieltjes algorithm, described in Section \ref{ssec:stieltjes} and summarized in Algorithm \ref{alg:ms}.
  \item (\MM) The moment method of Section \ref{ssec:moment}, involving direct orthogonalization of the monomial basis, i.e., the functions $\phi_j$ are monomials $x^\alpha$.
  \item (\ML) The moment method of Section \ref{ssec:moment}, involving direct orthogonalization of the tensorial Legendre polynomial basis, i.e., the functions $\phi_j$ are tensorial Legendre polynomials with respect to the uniform measure $\nu$ whose support is a bounding box of the support of $\mu$.
\end{itemize}
All methods assume the ability to compute (general) polynomial moments in order to carry out computations, and we describe in subsections below how we approximately or exactly compute these moments via quadrature. Each experiment uses the same quadrature rule for all methods. A summary of the experiments (i.e., the various measures $\mu$) is given in Table  \ref{tab:experiments}

\begin{table}[tbhp]
{\footnotesize
  \caption{Abbreviation and subsection for the examples considered.}\label{tab:experiments}
\begin{center}
  \begin{tabular}{@{}lccc@{}}\toprule
    Example &  Abbreviation & Dimension & Section \\\midrule
    Jacobi weight function & \Jacobi & 2, 3 & \ref{sssec:jacobi} \\
    Uniform measure on an annulus & \Annulus & 2 & \ref{sssec:annulus} \\
    Uniform measure between polar curves & \Curve & 2 & \ref{sssec:curve} \\ 
    Uniform Measure within a torus & \Torus & 3 & \ref{sssec:torus} \\
    Uniform measure on a rectangle minus a ball & \Hole & 2 & \ref{sssec:hole} \\ 
    Uniform measure on Madagascar & \Map & 2 & \ref{sssec:map} \\\bottomrule
  \end{tabular}
\end{center}
}
\end{table}

\subsection{Experiments with moments via tensorized Gaussian quadrature}\label{ssec:gq}
In this subsection we compute polynomial moments with respect to the measure $\mu$ via tensorized Gaussian quadrature. In all cases except Section \ref{sssec:curve}, this quadrature is exact (in infinite precision).

\subsubsection{\Jacobi: Tensorized Jacobi measure}\label{sssec:jacobi}
For an initial investigation, we consider a tensorial measure,
\begin{align*}
\dx{\mu}(x) &= \prod_{i=1}^d B(\alpha_i, \beta_i) (1-x_i)^{\alpha_i} (1+x_i)^{\beta_i}, & x &\in [-1, 1]^d & \alpha_i, \beta_i &> -1,
\end{align*}
where $B(\cdot,\cdot)$ is the Beta function, which is a tensorized Jacobi (or Beta) measure. To compute moments, we utilize tensorized Gauss quadrature of sufficiently high order so that all integrals are exact (in exact arithmetic). We randomly generated the $\alpha_i, \beta_i$ parameters by uniformly sampling over the interval $(-1, 10)$, resulting in the choices:
\begin{align*}
  (\alpha_1, \alpha_2) &= (3.80, 0.78), & (\beta_1, \beta_2) &= (7.34, 8.26), & d&= 2 \\
  (\alpha_1, \alpha_2, \alpha_3) &= (1.61, 0.32, 3.01), & (\beta_1, \beta_2, \beta_3) &= (-0.89, 9.83, 7.67), & d&= 3 
\end{align*}

\begin{figure}[htbp]
\centering
\includegraphics[width=1\textwidth]{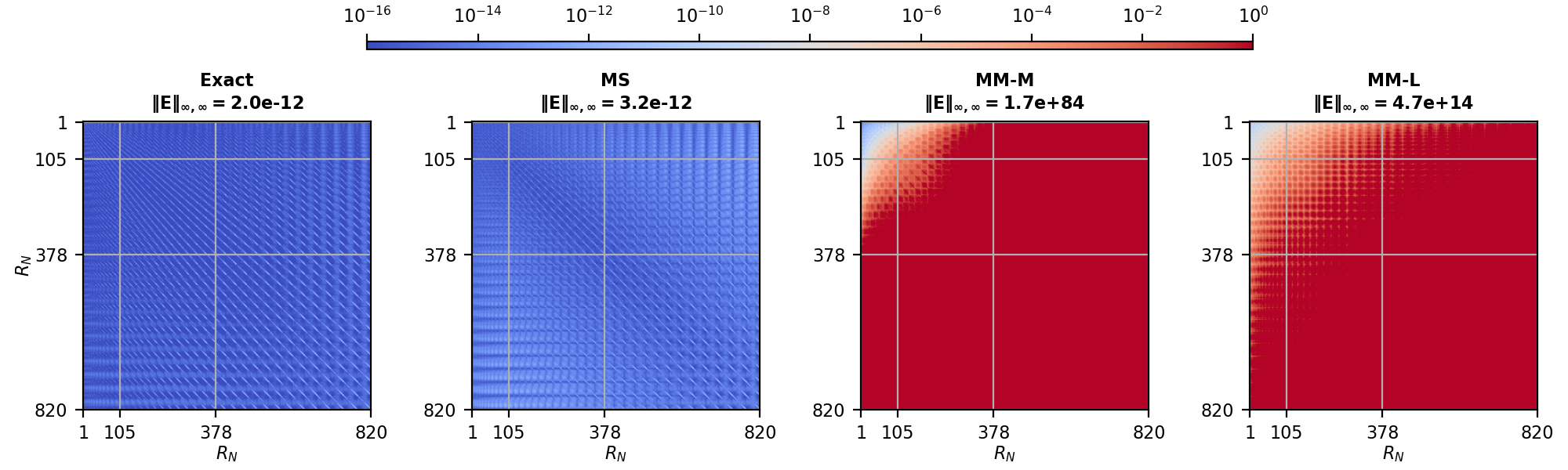}\\
\includegraphics[width=1\textwidth]{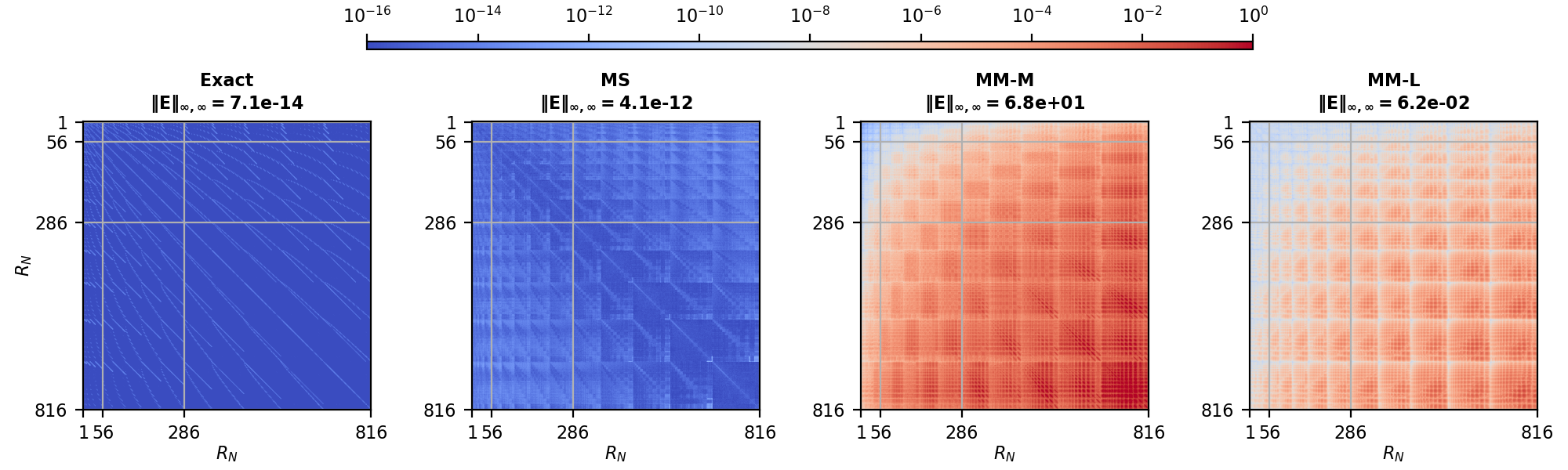}
  \caption{\Jacobi \ results, Section \ref{sssec:jacobi}: Visualization of the error matrix $E$ in dimension $d=2$ (top) and $d=3$ (bottom). From left to right in each row: the Exact, \MS, \MM, and \ML{} algorithms. We choose $10^0$ as the upper saturation point for all colormaps as values beyond this indicate $\mathcal{O}(1)$ error; we continue to impose this saturation value for all subsequent error plots.}
 \label{fig:tensor}
\end{figure}

Figure \ref{fig:tensor} visualizes the Gramian error matrix $E$ in \eqref{eq:e_N} for $d=2, 3$. One observes that the Exact and \MS{} algorithms performs very well, but both the \MM{} and \ML{} algorithms suffer accumulation of roundoff error. Even in this case, when the \ML{} algorithm uses a ``reasonable'' choice of basis for orthogonalization, instabilities develop quickly. In this simple case when quadrature is numerically exact, a standard orthogonalization routine produces unstable results.

\subsubsection{\Annulus: Measure on an Annulus}\label{sssec:annulus}
Our second test is the uniform measure $\mu$ with support in an annular region in $d=2$ dimensions centered at the origin. In polar coordinates $(r,\theta)$, this is represented as,
\begin{align}\label{eq:domain-polar}
  \mathrm{supp}(\mu) = \left\{ (r,\theta) \;\; \big| \theta_1 \leq \theta \leq \theta_2, \;\; r_1(\theta) \leq r \leq r_2(\theta) \right\},
\end{align}
where $(\theta_1, \theta_2) = (0, 2\pi)$ and $(r_1(\theta), r_2(\theta)) = (0.5, 1.0)$. Quadrature with respect to the uniform measure on this domain can exactly integrate polynomials (in the $x$ variable) using a tensor-product quadrature over $(r,\theta)$ space using Fourier quadrature in $\theta$ and Legendre-Gauss quadrature in the $r$ variable. We use a large enough quadrature rule so that all integrals are numerically exact.

\begin{figure}[htbp]
\centering
\includegraphics[width=0.6\textwidth]{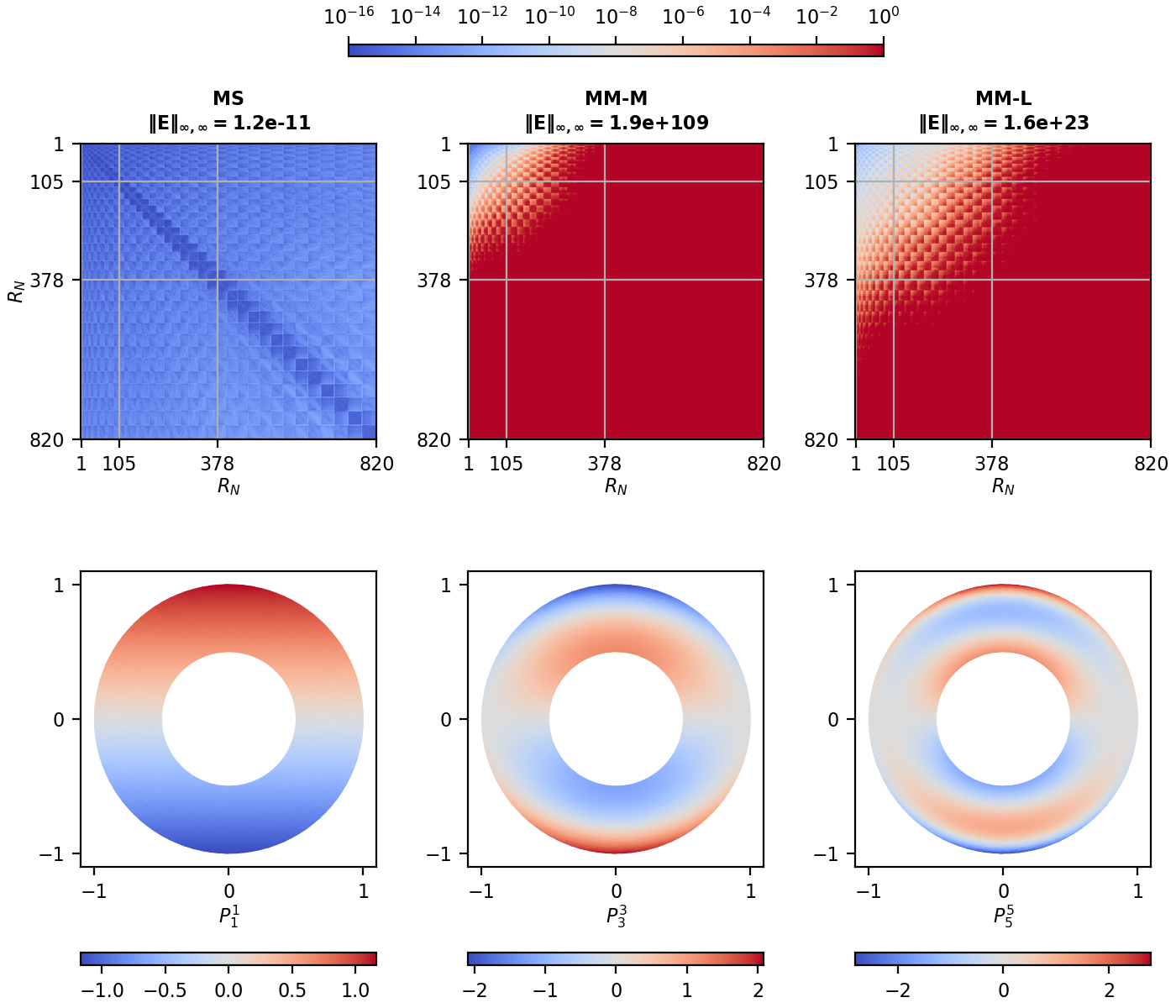}
  \caption{\Annulus \ results, Section \ref{sssec:annulus}. Top: Visualization of the error matrix $E$ for the \MS{}, \MM{}, and \ML{} algorithms from left to right. Bottom: Evaluations of the $r_n$'th entry of $\mathbbm{p}_n$ for degree $n = 1, 3, 5$ using the \MS{} algorithm.}\label{fig:annulus}
\end{figure}

In this case we do not know exact formulas for the recurrence matrices, so we rely on the metric $E$. Figure \ref{fig:annulus} shows again that the \MS{} algorithm performs better than the \MM{} and \ML{} methods for large degrees.

\subsubsection{\Curve: Measure within polar curves}\label{sssec:curve}
We now consider a more difficult example again in $d=2$ dimensions. We again use a curve defined as \eqref{eq:domain-polar}, but this time it is a region bounded between two Archimedean spirals. In particular, we set,
\begin{align*}
  (\theta_1, \theta_2) &= (0, 6\pi), & (r_1(\theta), r_2(\theta)) &= (0.8\theta, \theta).
\end{align*}
Again we choose $\mu$ as the uniform measure over the region defined in \eqref{eq:domain-polar}.

We write integrals as iterated, with the inner integral over $r$ exactly computable using Legendre-Gauss quadrature. But the outer integral in $\theta$ involves terms both polynomial and trigonometric polynomial in $\theta$, and we approximately integrate these values with a $10^6$-point Fourier quadrature rule. As can be seen in Figure \ref{fig:spiral}, the novel \MS{} procedure once again is much more stable than the \MM{} and \ML{} approaches.

\begin{figure}[htbp]
\centering
\includegraphics[width=0.6\textwidth]{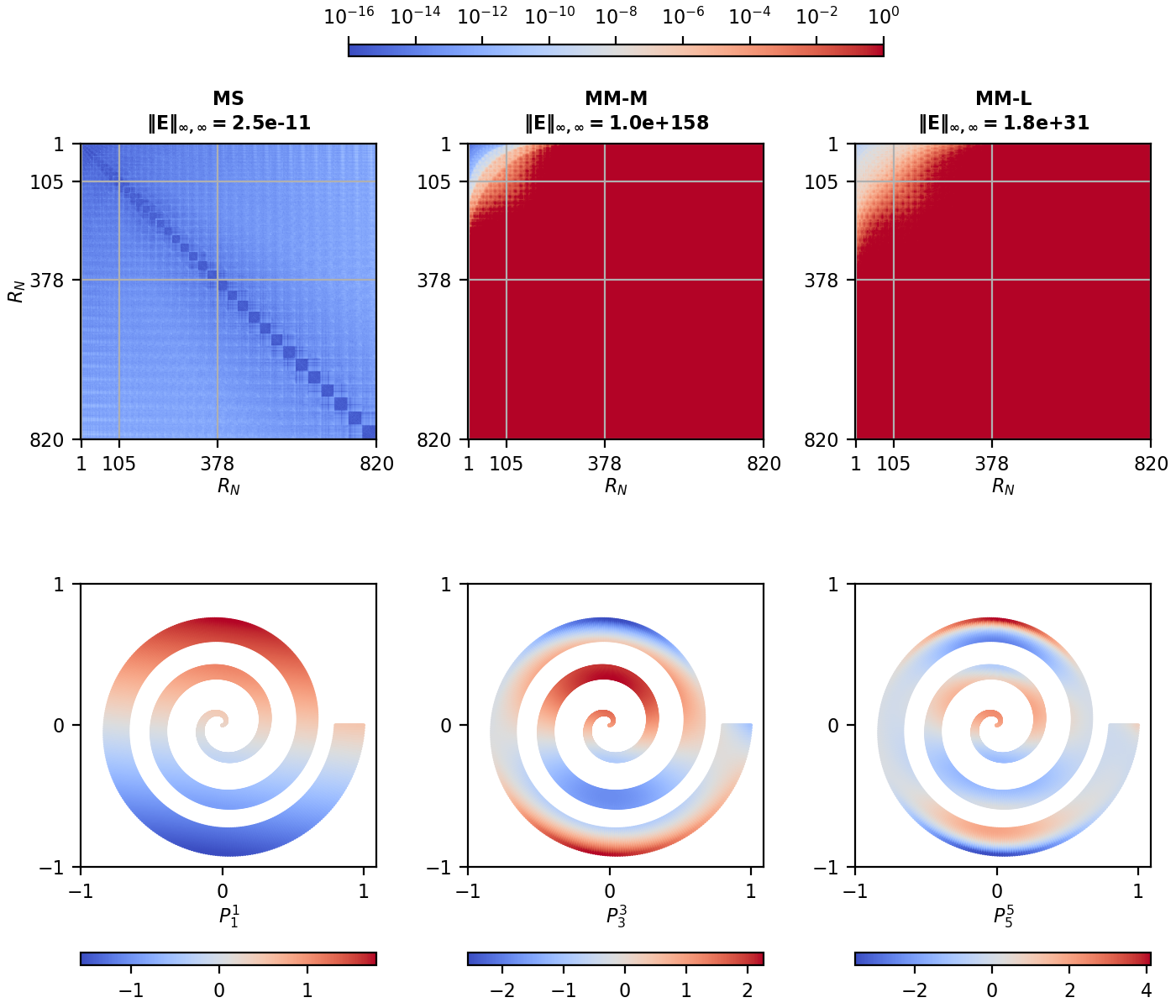}
  \caption{\Curve \ results, Section \ref{sssec:curve}.  Top: Visualization of the error matrix $E$ for the \MS{}, \MM{}, and \ML{} algorithms from left to right. Bottom: Evaluations of the $r_n$'th entry of $\mathbbm{p}_n$ for degree $n = 1, 3, 5$ using the \MS{} algorithm.}\label{fig:spiral}
\end{figure}

\subsubsection{\Torus: Uniform measure inside a torus}\label{sssec:torus}
We consider the uniform measure over the interior of a torus, whose parametric representation of the boundary is given by,
\begin{align*}
  x_1(\theta, \phi) &= (R + r\cos(\theta)) \cos{\phi}, &
  x_2(\theta, \phi) &= (R + r\cos(\theta)) \sin{\phi}, &
x_3(\theta, \phi) &= r \sin(\theta),
\end{align*}
for $\theta, \phi \in [0, 2\pi)$. The interior of the torus is defined by $(\sqrt{x_1^2+x_2^2} - R^2)^2 + x_3^2 < r^2$. We choose $r=1$ and $R=2$.

Quadrature with respect to the uniform measure on this domain can exactly integrate polynomials using a tensor-product quadrature over $(r,\theta,\phi)$ space using Fourier quadrature in $\theta$ and $\phi$ and an Legendre-Gauss quadrature in the $r$ variable. We use a large enough quadrature rule so that all integrals are numerically exact. In Figure \ref{fig:torus_1} we again observe that \MS{} outperforms \MM{} and \ML. However, these two moment-based procedures give more reasonable results in this case since the polynomial degree is relatively low.
\begin{figure}[htbp]
\centering
\includegraphics[width=0.6\textwidth]{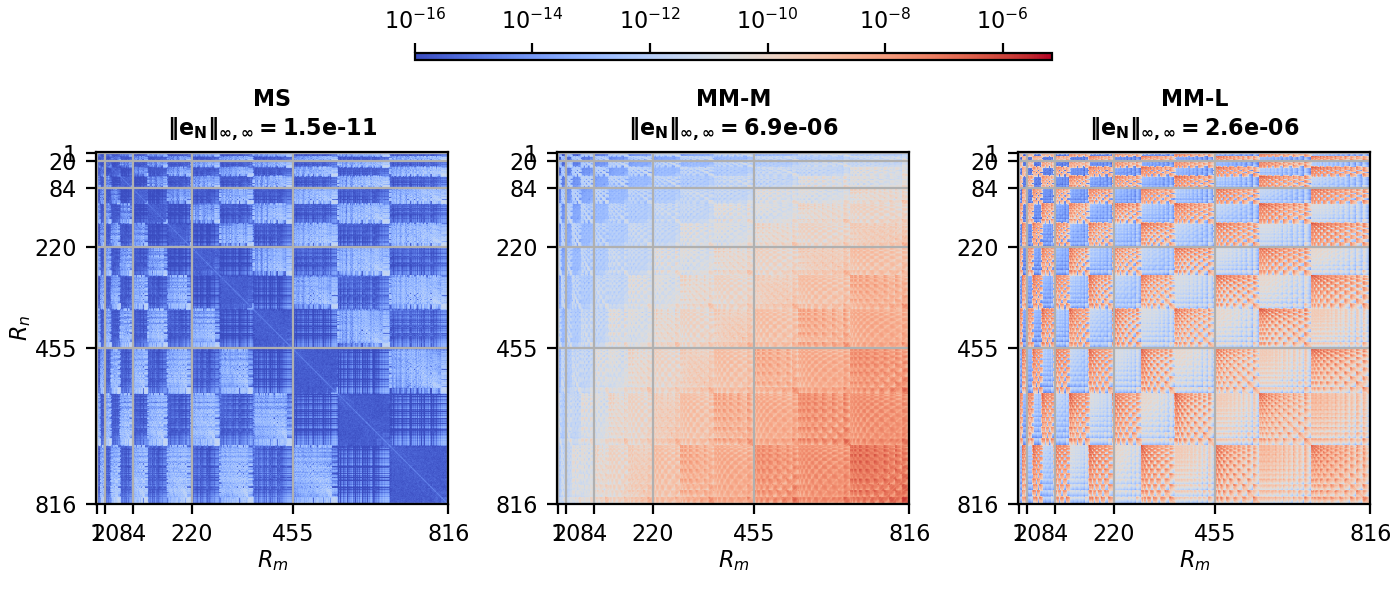}
  \caption{\Torus \ results, Section \ref{sssec:torus}. Visualization of the error matrix $E$ for the \MS{}, \MM{}, and \ML{} algorithms from left to right.}\label{fig:torus_1}
\end{figure}

\subsection{Moments via Monte Carlo techniques}\label{subsec:mc}
We now consider two more complicated domains, where integration is performed approximately using a Monte Carlo quadrature rule. I.e., we approximate moments with respect to a uniform measure over a domain $D$ via,
\begin{align*}
  \int_D p(x) \dx{x} \approx \frac{1}{M} \sum_{m=1}^M p(x_m),
\end{align*}
where $D$ is a two-dimensional domain and $\{x_m\}_{m=1}^M$ are iid random samples from $\mu$. In all examples we use a single, fixed instance of the Monte Carlo samples. Therefore one can consider this computing approximately orthogonal polynomials with respect to the uniform measure over $D$, or as computing (numerically) exactly orthogonal polynomials with respect to a size-$M$ discrete measure. We emphasize again that our goal is not to construct accurate quadrature rules, but rather to construct orthogonal polynomials given some quadrature rules. In all simulations here we take $M = 10^8$.

\subsubsection{\Hole: Square with a hole}\label{sssec:hole}
We consider the uniform measure $\mu$ over the two dimension domain $[-1,1]^2 \backslash B_1(0)$, where $B_1(0)$ is the origin-centered unit ball of radius 1. Figure \ref{fig:hole} shows results for this experiment. The \MS{} algorithm again performs the best, but we see a notably increased error in the orthogonality metric compared with the previous examples. We attribute this increased error to an increase in the condition number of the associated matrices of the \MS{}. We investigate this in more detail in Section \ref{ssec:condnum}.

\begin{figure}[htbp]
\centering
\includegraphics[width=0.6\textwidth]{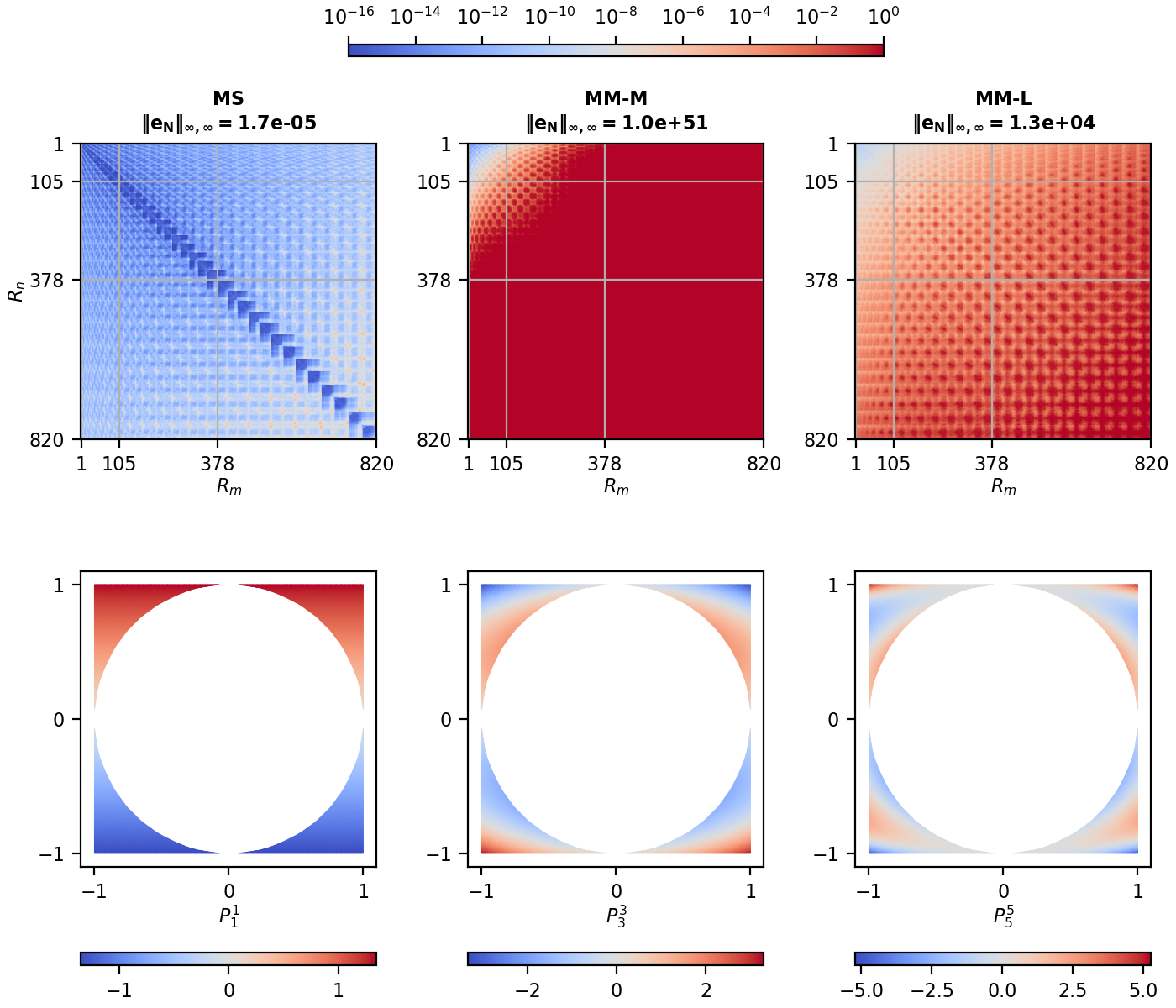}
\caption{\Hole \ results, Section \ref{sssec:hole}: Visualization of the error matrix $E$ for the \MS{}, \MM{}, and \ML{} algorithms from left to right. Bottom: Evaluations of the $r_n$'th entry of $\mathbbm{p}_n$ for degree $n = 1, 3, 5$ using the \MS{} algorithm.}\label{fig:hole}
\end{figure}

\subsubsection{\Map: Measure on map of Madagascar}\label{sssec:map}
Our final two-dimensional example is the region of the country of Madagascar. We draw random samples from this region via rejection sampling over a latitude-longitude bounding box, where the rule for inclusion in the domain is defined by positive elevation, which can be sampled via the data in \cite{smith1997global}. We map the bounding box to $[-1,1]^2$ for simpler plotting.

As can be seen in Figure \ref{fig:map}, our orthogonality metric explodes very quickly for the \MM{} and \ML{}, even for relatively small polynomial degree. The \MS{} succeeds to a much greater degree, but produces relatively large errors. Again we attribute this to increased ill-conditioning of associated matrices in our procedure, see the uptrend of condition number in Section \ref{ssec:condnum}.

\begin{figure}[htbp]
\centering
\includegraphics[width=0.6\textwidth]{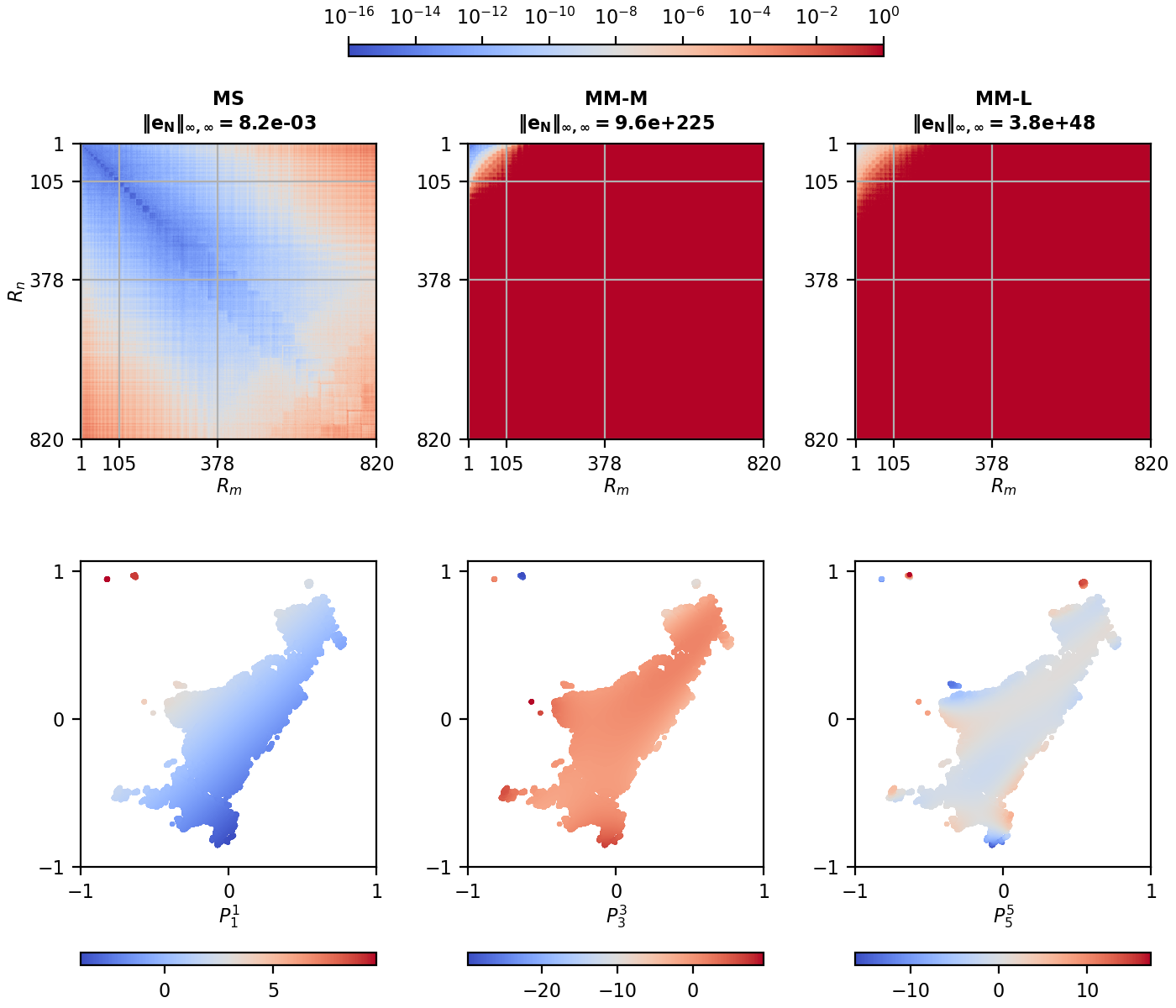}
  \caption{\Map \ results, Section \ref{sssec:map}: Top: Visualization of the error matrix $E$ for the \MS{}, \MM{}, and \ML{} algorithms from left to right. Bottom: Evaluations of the $r_n$'th entry of $\mathbbm{p}_n$ for degree $n = 1, 3, 5$ using the \MS{} algorithm.}\label{fig:map}
\end{figure}

\subsection{Stability investigation via condition numbers}\label{ssec:condnum}
We justify the accuracy of the \MS{}, \MM{}, and \ML{} algorithms by investigating the condition numbers of some of the associated matrices in each procedure. For the \MM{} and \ML{} algorithms, we investigate the condition number of the Gram matrix $G_n$ in \eqref{eq:gram-matrix}. For \MS{}, we investigate the condition number of the moment matrix $T_{n,i,i}$ in \eqref{eq:ST-def} plotting the average of condition number of $T_{n,i,i}$ over all $i$.

Figure \ref{fig:cond} shows these condition numbers for all our previous examples. We note that the $G_n$ matrices are badly ill-conditioned for larger degrees, but the $T_n$ matrices are much better conditioned. In addition, we see that for the \Map{} and \Hole{} cases the condition number of $T_n$ is larger than for other cases, which motivates why even the \MS{} algorithm struggles for these domains.

\begin{figure}[htbp]
  \begin{center}
\includegraphics[width=1.00\textwidth]{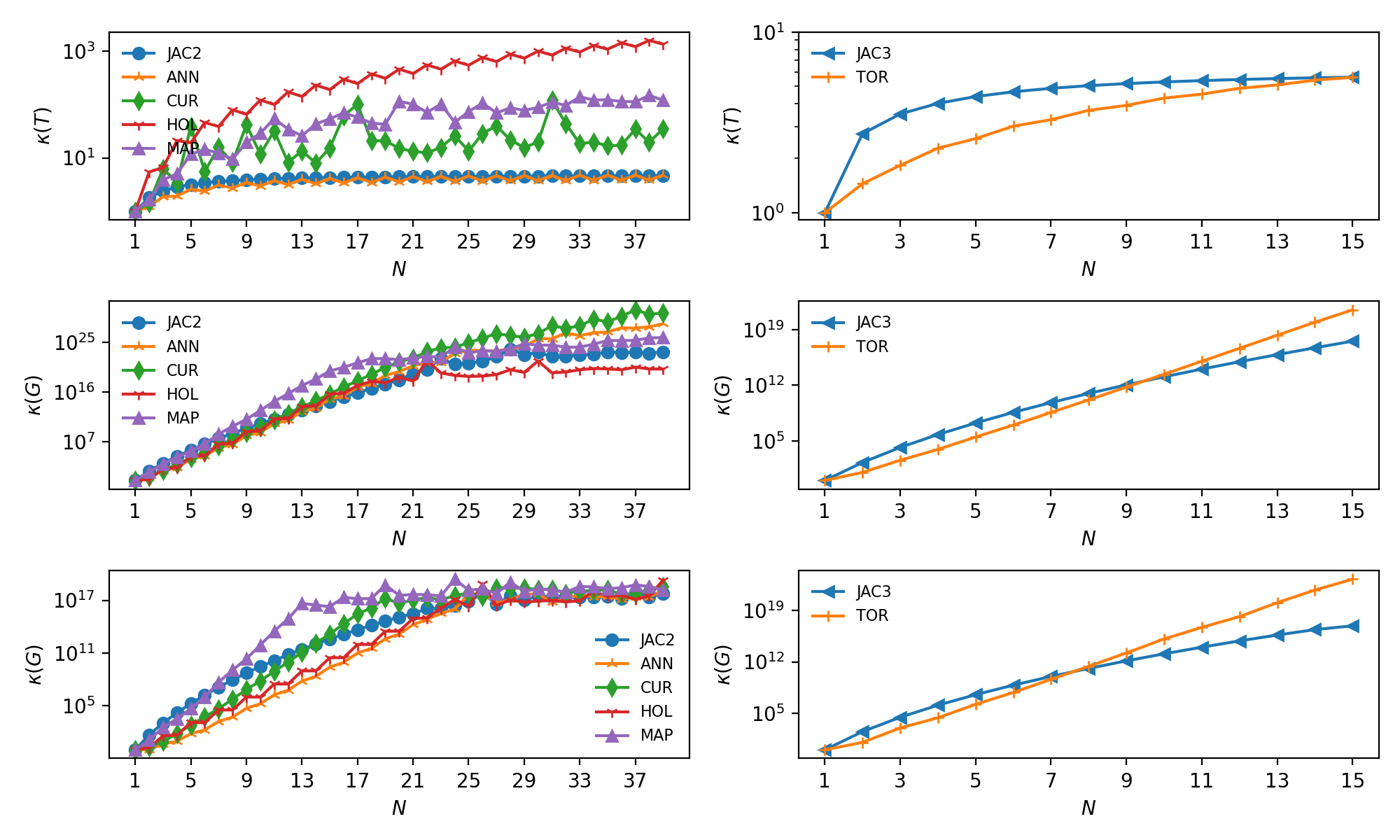}
  \end{center}
  \caption{Top row: Condition numbers of the moment matrices $T_{n,i,i}$ in \eqref{eq:ST-def} used in the \MS{} algorithm (averaged over $i \in [d]$). Middle row: Condition numbers for the Gram matrices $G_n$ in \eqref{eq:gram-matrix} used in the \MM{} algorithm. Bottom row: Same as the middle row but for the \ML{} algorithm.}\label{fig:cond}
\end{figure}

\subsection{The Christoffel function}\label{ssec:christoffel}
The ability to stably compute an orthogonal basis in multiple dimensions allows us to investigate interesting phenomena. Let $\mu$ be uniform over a compact set in $\numset{R}^d$, and consider the diagonal $K_N$ of the degree-$N$ normalized reproducing kernel, and its inverse $\lambda_N$, the Christoffel function,
\begin{align*}
  K_N(x) &= \frac{1}{R_N} \sum_{n=0}^N \mathbbm{p}_n^T(x) \mathbbm{p}_n(x),  &
  \lambda_N(x) &= 1/K_N(x),
\end{align*}
so that $K_N(x) \dx{\mu}(x)$ is a probability density. Random sampling from this probability measure is known to result in near-optimal sample complexity for constructing least-squares polynomial approximations to functions in the $L^2_\mu$ norm \cite{cohen_optimal_2017}. Plotting such densities is itself interesting, but even more so is that fact that as $N \uparrow \infty$, such densities weakly converge to the Monge-Amp\`{e}re measure over $\mathrm{supp}(\mu)$ \cite{berman_growth_2010,levenberg_weighted_2010}, which is a fundamental quantity of theoretical interest in polynomial approximation in several variables \cite{levenberg_approximation_2006,bloom_polynomial_2012}. As analytic forms for such measures are unknown for general domains, it is interesting to use numerical algorithms to investigate the finite but large $N$ behavior of $K_N$, which is not possible directly without the ability to stably compute an orthonormal basis. Figure \ref{fig:christoffel} plots both $K_N$ and $\lambda_N$ for four of our two-dimensional domains with $N=39$.

\begin{figure}[htbp]
\centering
\includegraphics[width=1\textwidth]{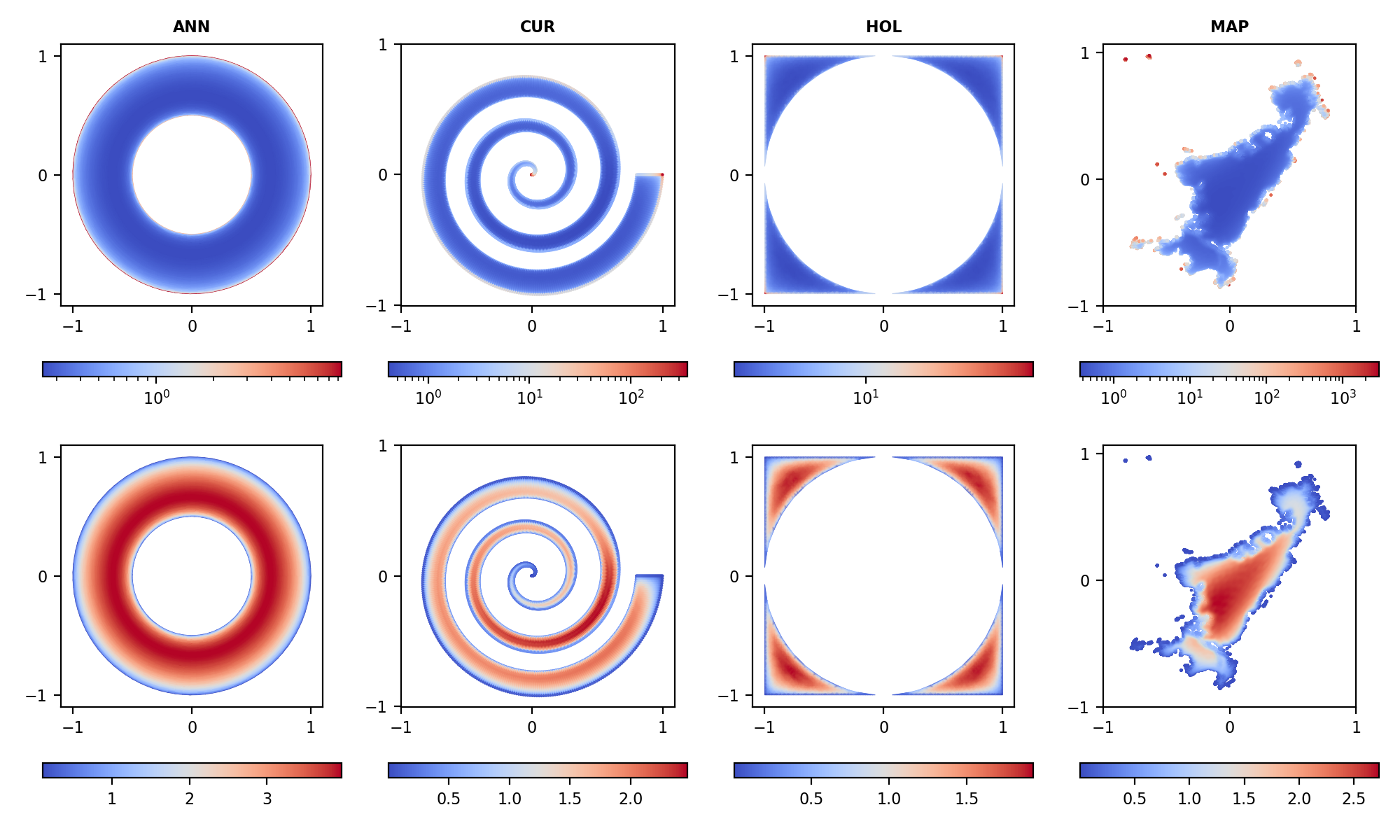}
  \caption{The normalized reproducing diagonal $K_N$ (top row) and normalized Christoffel function $\lambda_N$ (bottom row) for $N=39$.}\label{fig:christoffel}
\end{figure}

\section{Conclusions}\label{sec:conclusions}
In this paper, we extend existing approaches for computing recurrence coefficients from the univariate case to the multivariate case. We propose a new, Multivariate Stieltjes (\MS) algorithm for computing recurrence matrices that allows stable evaluation of multivariate orthonormal polynomials. 

We demonstrate with several numerical examples the substantially improved stability of the new algorithm compared to direct orthogonalization approaches. For both small dimension and small polynomial degree, there is little benefit, but \MS{} outperforms other methods when one requires polynomials of moderately large degree.

The algorithm is essentially explicit in two and three dimensions, but requires the numerical solution to a non-convex optimization problem in more than three dimensions, whose investigation would be a natural extension of this work.

\bibliographystyle{siam}
\bibliography{references.bib}

\appendix
\section{Algorithms}\label{sec:alg}

\begin{algorithm}
\caption{Algorithm via three-term relation}\label{alg:ttr}
\KwIn{dimension $d$ and univariate recurrence coefficients $\{a_{i,n}, b_{i,n}\}_{n=0}^N$, $i\in[d]$}
\For{$n=1$ \KwTo $N$}{
	\For{$i=1$ \KwTo $d$}{
		compute $c\left(\alpha^{(n,k)}, i\right)$ defined in Section \ref{ssec:recurrence} \;
		compute $A_{n,i}, B_{n,i}$ by \eqref{eq:AB-tensorial}
	}
}
\KwOut{coefficient matrices $\{A_{n,i}, B_{n,i}\}_{n=0}^N$}
\end{algorithm}

\begin{algorithm}
\caption{Stieltjes procedure}\label{alg:ms}
\SetAlgoVlined
\KwIn{dimension $d$, max degree $N$, measure $\dx{\mu}$ on support $\Omega$.}
\For{$n=0$ \KwTo $N-1$}{
\For{$i=1$ \KwTo $d$}{
evaluate orthonormal polynomials $\mathbbm{p}_{n}$ by \eqref{eq:diagonal-ttr} \;
compute $A_{n+1,i}$ using moment matrices (\eqref{eq:An-stieltjes}) \;
evaluate the modified polynomials basis $\widetilde{\mathbbm{p}}_{n,i}$ by \eqref{eq:ptilde-def} \;
compute symmetric moments $T_{n,i,i}$ by \eqref{eq:ST-def} \;
compute $U_{n+1,i}$ and $\Sigma_{n+1,i}$ in the SVD of $B_{n+1,i}$ from \eqref{eq:stieltjes-symmetric} \;
\For{$j=i+1$ \KwTo $d$}{
compute the mixed moment $T_{n,i,j}$ by \eqref{eq:ST-def} \;
}
}
determine $\widehat{V}_{n+1,1}$ and $\widetilde{V}_{n+1,i}$ by \eqref{eq:V_1}, and assemble the matrix $V_{n+1, 1}$ by \eqref{eq:V-def} \;
compute $B_{n+1,1}$ by \eqref{eq:B-decomposition} \;
\For{$j=2$ \KwTo $d$}{
compute $\widehat{V}_{n+1,j}$ by \eqref{eq:hatV} \;
compute $\bs{y} \bs{y}^T$ by \eqref{eq:V-orthonormality} \;
\uIf{d=2}{
compute $\bs{y}$ by \eqref{eq:yz-condition}
} \uElseIf {d=3}{
\eIf{n=0}{
compute $B_{n+1,j}$ following the strategy in Section \ref{sssec:backonmoments}
}{
compute $K_{n+1,j}$ from \eqref{eq:Kj-def} and its kernel $\Psi_j$ \;
compute $E_j$ from \eqref{eq:Cj-decomp} \;
\eIf{j=2}{
Set $W_2 = I_{\Delta_{r_n+1}}$}{
compute $W_3$ from \eqref{eq:E3-wopp} following the strategy in Section \ref{sssec:d3-woop}}
}
assemble the matrix $V_{n+1,j}$ by \eqref{eq:V-def} \;}
\uElse{solve $W_j$ from \eqref{eq:wopp}}
determine $C_j$ by \eqref{eq:Cj-decomp}, and thus $\widetilde{V}_{n+1,j} by \eqref{eq:Cj-def}$
}
compute $B_{n+1,j}$ by \eqref{eq:B-decomposition} \;}
\KwOut{coefficient matrices $\{A_{n,i}, B_{n,i}\}_{n=0}^N$}
\end{algorithm}

\end{document}